\documentclass{amsart}

\usepackage[latin1]{inputenc}
\usepackage{subfigure}
\usepackage{hyperref}

\usepackage{amscd,amssymb,amsmath,tikz}
\usepackage{enumerate}
\usepackage{enumitem}
\usepackage{mathrsfs}
\usepackage{appendix}
\usepackage{pdflscape}
\usepackage{float}

 \newtheorem{theorem}{Theorem}[section]
 \newtheorem{cor}[theorem]{Corollary}
 \newtheorem{lemma}[theorem]{Lemma}
 \newtheorem{proposition}[theorem]{Proposition}
  
    \newtheorem{ex}[theorem]{Example}
 \theoremstyle{definition}
 \newtheorem{definition}[theorem]{Definition}
 \theoremstyle{remark}
 \newtheorem{remark}[theorem]{Remark}
 \numberwithin{equation}{section}

\DeclareMathOperator{\di}{\mathcal{D}}
\DeclareMathOperator{\dis}{\mathfrak{D}}
\newcommand{\rdis}{\mathfrak{D}^{{\rm rev}}}
\newcommand{\rdi}{\mathcal{D}^{{\rm rev}}}

\DeclareMathOperator{\fs}{\mathfrak{F}}

\DeclareMathOperator{\yqk}{\mathfrak{Y}}
\DeclareMathOperator{\qkey}{\mathfrak{Q}}
\DeclareMathOperator{\yqs}{\mathscr{S}}
\DeclareMathOperator{\qs}{\mathcal{S}}
\newcommand{\yfs}{\mathfrak{F}^{Y}}

\DeclareMathOperator{\des}{DES}
\DeclareMathOperator{\wdes}{{\mathrm wdes}}

\DeclareMathOperator{\std}{{\mathrm std}}
\DeclareMathOperator{\rev}{{\mathrm rev}}

\DeclareMathOperator{\QSym}{{QSym}}
\DeclareMathOperator{\NSym}{{NSym}}
\DeclareMathOperator{\Sym}{{Sym}}
\DeclareMathOperator{\im}{\mathfrak{S}}

\newcommand{\SSRIF}{\mathrm{SSRIF}}
\newcommand{\SRIF}{\mathrm{SRIF}}
\newcommand{\SSIF}{\mathrm{SSIF}}
\newcommand{\SIF}{\mathrm{SIF}}

\newcommand{\RSSF}{\mathrm{RSSF}}

\newcommand{\RSF}{\mathrm{RSF}}

\newcommand{\YSSF}{\mathrm{YSSF}}
\newcommand{\YSF}{\mathrm{YSF}}
\newcommand{\FSSF}{\mathrm{FSSF}}

\newcommand{\YCT}{\mathrm{YCT}}
\newcommand{\SYCT}{\mathrm{SYCT}}
\newcommand{\SSIT}{\mathrm{SSIT}}
\newcommand{\SIT}{\mathrm{SIT}}
\newcommand{\SSRIT}{\mathrm{SSRIT}}
\newcommand{\SRIT}{\mathrm{SRIT}}

\newcommand{\wt}{{\rm wt}}

\newcommand{\flatten}{{\rm flat}}

\newcommand{\excise}[1]{}%{$\star$\textsc{#1}$\star$}

\newcommand\tableau[1]{
\vcenter{
\let\\=\cr
\baselineskip=-16000pt
\lineskiplimit=16000pt
\lineskip=0pt
\halign{&\tableaucell{##}\cr#1\crcr}}}

\newcommand{\tableaucell}[1]{{%
\def \arg{#1}\def \void{}%
\ifx \void \arg
\vbox to \cellsize{\vfil \hrule width \cellsize height 0pt}%
\else
\unitlength=\cellsize
\begin{picture}(1,1)
\put(0,.22){\makebox(1,1)[b]{$#1$}}
\put(0,0){\line(1,0){1}}
\put(0,1){\line(1,0){1}}
\put(0,0){\line(0,1){1}}
\put(1,0){\line(0,1){1}}
\end{picture}%
\fi}}
\newlength\cellsize 

\savebox2{%
\begin{picture}(12,12)
\put(0,0){\line(1,0){12}}
\put(0,0){\line(0,1){12}}
\put(12,0){\line(0,1){12}}
\put(0,12){\line(1,0){12}}
\end{picture}}

\newcommand\cellify[1]{\def\thearg{#1}\def\nothing{}%
\ifx\thearg\nothing\vrule width0pt height\cellsize depth0pt%
  \else\hbox to 0pt{\usebox2\hss}\fi%
  \vbox to 12\unitlength{\vss\hbox to 12\unitlength{\hss$#1$\hss}\vss}}

\newcommand\vtableau[1]{\vtop{\let\\=\cr
\setlength\baselineskip{-12000pt}
\setlength\lineskiplimit{12000pt}
\setlength\lineskip{0pt}
\halign{&\cellify{##}\cr#1\crcr}}}

\usepackage[colorinlistoftodos]{todonotes}

\begin{document}

\date{April 6, 2020}

\title{Lifting the dual immaculate functions}

\author[S. Mason]{Sarah Mason}
\address{Department of Mathematics, 127 Manchester Hall, Wake Forest University, Winston-Salem, NC 27109, U.S.A.}
\email{masonsk@wfu.edu}

\author[D. Searles]{Dominic Searles}
\address{Department of Mathematics and Statistics, University of Otago, 730 Cumberland St., Dunedin 9016, New Zealand}
\email{dominic.searles@otago.ac.nz}


%%% ----------------------------------------------------------------------
\maketitle
%%% ----------------------------------------------------------------------
%\tableofcontents
\maketitle

\begin{abstract} 
We introduce two lifts of the dual immaculate quasisymmetric functions to the polynomial ring. We establish positive formulas for expansions of these dual immaculate slide polynomials into the fundamental slide and quasi-key bases for polynomials.  These formulas mirror connections between dual immaculate quasisymmetric functions, fundamental quasisymmetric functions, and Young quasisymmetric Schur functions, extending these connections from the ring of quasisymmetric functions to the full polynomial ring.  
We also consider a reverse variant of the dual immaculate quasisymmetric functions, mirroring the dichotomy between the quasisymmetric Schur functions and the Young quasisymmetric Schur functions. We show this variant is obtained by taking stable limits of one of our lifts, and utilize these reverse dual immaculate quasisymmetric functions to establish a connection between the dual immaculate quasisymmetric functions and the Demazure atom basis for polynomials.
\end{abstract}

%\tableofcontents

\section{Introduction}

The \emph{Schur functions} form a celebrated and much-studied basis for the Hopf algebra $\Sym$ of symmetric functions with key applications to many different areas of mathematics, including Schubert calculus of Grassmannian varieties and representation theory of both the symmetric and general linear groups.  The Hopf algebra $\QSym$ of \emph{quasisymmetric functions} \cite{Ges84} is a generalization of $\Sym$, containing $\Sym$ as a subalgebra.  Bases for $\QSym$ are indexed by compositions $\alpha$ (finite sequences of positive integers), whereas bases for $\Sym$ are indexed by partitions (weakly decreasing finite sequences of positive integers).  A classical basis for $\QSym$ is the \emph{fundamental quasisymmetric functions} $\{F_\alpha\}$, introduced in~\cite{Ges84} as generating functions for $P$-partitions.  In comparison to the role Schur functions play as Frobenius characters of the irreducible representations of the symmetric group, the fundamental quasisymmetric functions are quasisymmetric characteristics \cite{DKLT} of the irreducible representations of the $0$-Hecke algebra. There has been significant recent interest in finding other bases for $\QSym$ that also extend properties of the Schur basis. Two central examples of such \emph{Schur-like} bases of $\QSym$ are the \emph{quasisymmetric Schur} basis $\{\qs_\alpha\}$ introduced in \cite{HLMvW11a} and the \emph{dual immaculate} basis $\{\di_\alpha\}$ introduced in \cite{BBSSZ14}.  

A precise sense in which bases for $\QSym$ are Schur-like comes from considering the Hopf algebra $\NSym$ of \emph{noncommutative symmetric functions}~\cite{GKLLRT95, Haz03, Haz05}: a noncommutative analogue of $\Sym$, which is dual to $\QSym$.  A basis $\{B_\alpha\}$ for $\NSym$ is termed \emph{Schur-like} if the image of each basis element $B_\alpha$ under the natural projection from $\NSym$ to $\Sym$ is the Schur function $s_\alpha$, whenever $\alpha$ is a partition. The dual in $\NSym$ of the quasisymmetric Schur functions is the \emph{nonsymmetric Schur functions} \cite{BLvW}, and the dual of the dual immaculate functions (from which the dual immaculate quasisymmetric functions originated) is the  \emph{immaculate functions} \cite{BBSSZ14}, both of which are Schur-like bases of $\NSym$. These bases for $\QSym$ and $\NSym$ have been the focus of much recent work, such as~\cite{BLvW, HLMvW11b, LauMas10, BerSanZab13, BesvWi13, BBSSZ15, TewvWi15, Koenig, TvW:2}.

Both the dual immaculate quasisymmetric functions and the quasisymmetric Schur functions expand positively in the fundamental basis (\cite{BBSSZ14}, respectively \cite{HLMvW11a}), as do the Schur functions themselves.  Moreover, a positive formula for expanding the dual immaculate quasisymmetric functions into the \emph{Young quasisymmetric Schur functions} $\{\yqs_\alpha\}$ (a basis for $\QSym$ closely related to the quasisymmetric Schur functions: see, e.g., \cite{LMvWbook}) was recently established in~\cite{AHM18}.

There has also been significant recent interest in the lifting of bases and structures in symmetric and quasisymmetric functions to bases for the algebra of (general) polynomials.  A focus is on developing the combinatorics of these new bases for polynomials and understanding how they relate to known bases.  One particular motivation for this program is to uncover connections to the \emph{Schubert basis} for polynomials which could provide insight towards the long-standing open problem of describing the coefficients appearing in products of Schubert polynomials.

A lifting of the fundamental basis to a basis for polynomials called the \emph{fundamental slide} basis $\{\fs_a\}$ was introduced in~\cite{Assaf.Searles:1}; the fundamental slide basis contains the fundamental quasisymmetric polynomials, and limits to the fundamental quasisymmetric functions.  This was followed by a lifting of the quasisymmetric Schur basis for $\QSym$ to form the \emph{quasi-key} basis $\{\qkey_a\}$ for polynomials in \cite{Assaf.Searles:2}, whose elements expand positively in the fundamental slide basis. The \emph{key polynomials} (well-studied Demazure characters in Type A~\cite{RS95, LasSch90}), which are the analogous lifting of the Schur functions to polynomials, expand positively into the quasi-key polynomials. Moreover, the \emph{extended Schur function} basis for $\QSym$, which is dual to the \emph{shin functions} \cite{CFLSX} (a third Schur-like basis for $\NSym$) also has a polynomial lifting termed the \emph{lock polynomials} \cite{Assaf.Searles:3}. 
See \cite{Pechenik.Searles:2} for a survey of recent developments and progress in this program.

Motivated by this perspective, and by the results of \cite{BBSSZ14} and \cite{AHM18} providing positive formulas for expansions of dual immaculate quasisymmetric functions into important bases for $\QSym$, in this paper we introduce two polynomial lifts of the dual immaculate quasisymmetric functions, which we term the \emph{reverse dual immaculate slide polynomials} $\{\rdis_a\}$ and the \emph{(Young) dual immaculate slide polynomials} $\{\dis_a\}$.  This completes the picture for polynomial lifts of (the duals of) the three Schur-like bases for $\NSym$ discussed above (termed the ``canonical Schur-like bases of $\NSym$'' in \cite{Campbell}). We investigate properties of these new polynomial bases and their relationships to other known bases. The bases we introduce and relationships we establish are summarized in Figure~\ref{fig:qsym} (for quasisymmetric functions) and Figure~\ref{fig:poly} (for polynomials).

\begin{figure}[h]
\begin{center}
\begin{picture}(360,50)
\put(0,40){${\color{gray} \di_\alpha   \longrightarrow \;  \textrm{ (expands positively into) } \longrightarrow \; \yqs_\alpha  \longrightarrow \textrm{ (expands positively into) } \longrightarrow  F_\alpha} $}
\put(2,20){${\updownarrow}$}
\put(182,20){${\color{gray} \updownarrow}$}
\put(355,20){${\color{gray} \updownarrow}$}
\put(0,0){${\rdi_\alpha \longrightarrow \textrm{ (expands positively into) } \longrightarrow} \;\; {\color{gray} \qs_\alpha  \longrightarrow \textrm{ (expands positively into) } \longrightarrow  F_\alpha} $}
\end{picture}
\end{center}
\caption{Here we depict onnections between various bases for $\QSym$, where the vertical arrows represent the passage between Young (top row) and reverse (bottom row) families. Bases introduced and relationships established in this paper are put into context by including previously known information in {\color{gray}gray}.}{\label{fig:qsym}}
\end{figure}

To facilitate this investigation and establish our results, we repeatedly employ a passage between what we term ``reverse'' and ``Young'' versions of families of polynomials and quasisymmetric functions, which may be expressed combinatorially in terms of tableau fillings whose entries decrease (respectively, increase) along rows. This is exemplified in the distinction between the quasisymmetric Schur functions (a reverse family) and the Young quasisymmetric Schur functions (a Young family), where one is obtained from the other by reversing both the composition and the indices of the variable set. The dual immaculate quasisymmetric functions are naturally a Young family; we term the analogous reverse family the \emph{reverse dual immaculate quasisymmetric functions} $\{\rdi_\alpha\}$.

Although this reversal procedure might seem innocuous at first glance, certain concepts and properties are in fact more naturally stated and applied in one of the reverse or Young paradigms than the other, or indeed only manifest in one.  For example, in Theorem~\ref{thm:stablelimit} below we prove that stable limits of reverse dual immaculate slide polynomials exist and are always equal to reverse dual immaculate quasisymmetric functions.  However, one does not obtain analogous stability results for (Young) dual immaculate slide polynomials and (Young) dual immaculate quasisymmetric functions, thus necessitating working in the reverse paradigm for stability results; see Remark~\ref{rmk:noyoungstability}. We note additionally that all combinatorial bases for (general) polynomials we are aware of fall into the reverse paradigm, as demonstrated by \cite[Theorem C]{Sea}. Yet on the other hand, authors have opted to work with Young versions of quasisymmetric families that were originally introduced as reverse families (e.g., \cite{Li15}, \cite{MN-SkewRS}), due to, for example, the closer connection to established theory of symmetric functions enjoyed by the Young versions.

We prove that the reverse dual immaculate slide polynomials form a basis for the polynomial ring, and we establish a positive formula for their expansion in the fundamental slide basis for polynomials. 

\begin{theorem}\label{thm:rdistoslide}
Let $a$ be a weak composition. Then
\[\rdis_{a} = \sum_{S\in \SRIF(a)} \fs_{\wdes_{\SRIF}(S)},\]
where the sum is over \emph{standard reverse immaculate fillings} of shape $a$ (Definition~\ref{def:SSRIF}), and the fundamental slide polynomials are indexed by the \emph{weak descent compositions} (Definition~\ref{def:wdesrdis}) of these fillings.
\end{theorem}

One can define a \emph{stable limit} for reverse families of polynomials, via prepending $m$ zeros to the weak composition indexing the polynomial and letting $m\to \infty$. We prove that the stable limits of reverse dual immaculate slide polynomials are exactly the reverse dual immaculate quasisymmetric functions $\{\rdi_a\}$. 

\begin{theorem}\label{thm:stablelimit}
Let $a$ be a weak composition. Then 
\[\lim_{m\to \infty} \rdis_{0^m\times a} = \rdi_{\flatten(a)},\]
where $0^m\times a$ denotes the weak composition formed by prepending $m$ zeros to $a$, and $\flatten(a)$ is the composition obtained by removing all zero parts of $a$.
\end{theorem}

As a corollary, our Theorem~\ref{thm:rdistoslide} stabilizes to, and thus extends, a tableau formula for expanding reverse dual immaculate quasisymmetric functions in the fundamental basis for $\QSym$; this tableau formula is the reverse analogue of the formula of \cite{BBSSZ14} for the fundamental expansion of a dual immaculate quasisymmetric function.

We also provide the first formula for quasi-key polynomials in terms of \emph{standard} tableau objects, in contrast to the existing formulas that are all in terms of semistandard objects. This standard formula is needed for our proof, via an insertion algorithm, that reverse dual immaculate slide polynomials expand positively in the quasi-key polynomials.

After establishing these results in the reverse perspective, we pivot our focus to the Young perspective in order to make use of and extend results and theory already developed for quasisymmetric functions. To this end, we introduce and utilize a Young version $\{\yqk_a\}$ of the quasi-key polynomials, lifting the Young quasisymmetric Schur basis for $\QSym$, and a Young version $\{\yfs_a\}$ of the fundamental slide polynomials, which provides an alternative lifting of the fundamental basis. We utilize Theorem~\ref{thm:rdistoslide} to obtain a formula for the Young fundamental slide expansion of the dual immaculate slide polynomials. We show that the finite-variable version of the formula of \cite{BBSSZ14} for the fundamental expansion of a dual immaculate quasisymmetric function is obtained as special case of our formula, hence it lifts the expansion formula of \cite{BBSSZ14} from $\QSym$ to the polynomial ring.

We extend an insertion algorithm of \cite{AHM18} from tableaux of composition shape to fillings of weak composition shape, and apply this \emph{weak insertion} algorithm to prove a positive formula for the expansion of dual immaculate slide polynomials into Young quasi-key polynomials.

\begin{theorem}\label{thm:distoyqk}
Let $a$ be a weak composition. Then
\[\dis_a = \sum_b c_{a,b} \yqk_b,\]
where $c_{a,b}$ is the number of \emph{dual immaculate recording fillings} of shape $b$ with \emph{row strip shape} $\rev(a)$ (Definition~\ref{def:DIRF}).
\end{theorem}

Theorem~\ref{thm:distoyqk} lifts the formula of \cite{AHM18} for the expansion of dual immaculate quasisymmetric functions into Young quasisymmetric Schur functions, from $\QSym$ to the full polynomial ring.

Finally, we return our attention to the reverse paradigm. This allows us to establish a connection to the well-studied \emph{Demazure atom} basis for polynomials~\cite{LasSch90},~\cite{Mas09}, which are realized both as the $t=q=\infty$ specialization of nonsymmetric Macdonald polynomials and, additionally, as characters of quotients of Demazure modules. We use Theorem~\ref{thm:distoyqk} to give a positive formula for expanding reverse dual immaculate slide polynomials into quasi-key polynomials, which further establishes, by a result of~\cite{Sea}, that the reverse dual immaculate slide polynomials expand positively in  Demazure atoms. Taking stable limits of this formula and applying Theorem~\ref{thm:stablelimit}, we obtain that reverse dual immaculate quasisymmetric functions expand positively in quasisymmetric Schur functions, and thus expand positively in Demazure atoms. In particular, this establishes that the dual immaculate quasisymmetric functions of \cite{BBSSZ14} expand positively in Demazure atoms up to a reversal of variables.

The polynomial bases we introduce in this paper and the relationships we establish between them are illustrated in Figure~\ref{fig:poly}. Each basis in Figure~\ref{fig:poly} is a lifting to polynomials of the basis for quasisymmetric functions in the same location in Figure~\ref{fig:qsym}.

\begin{figure}[h]
\begin{center}
\begin{picture}(360,50)
\put(0,40){${\dis_a} \;\;  {\longrightarrow} \;   \textrm{ {(expands positively into)} } \; {\longrightarrow} \;\; { \yqk_a} \;  {\longrightarrow} \; \textrm{ {(expands positively into)} } {\longrightarrow} { \;\; \yfs_a} $}
\put(2,20){${\updownarrow}$}
\put(180,20){${\updownarrow}$}
\put(352,20){${\updownarrow}$}
\put(0,0){${\rdis_a \longrightarrow \textrm{ (expands positively into) } \longrightarrow} {\color{gray} \; \qkey_a} \; {\color{gray}\longrightarrow  \textrm{(expands positively into) } \longrightarrow \; \fs_a} $}
\end{picture}
\end{center}
\caption{Here we depict connections between lifts of the bases for $\QSym$ in Figure~\ref{fig:qsym} to bases for the polynomial ring, where the vertical arrows represent the passage between Young (top row) and reverse (bottom row) families. Bases introduced and relationships established in this paper are put into context by including previously known information in {\color{gray}gray}. }{\label{fig:poly}}
\end{figure}

\subsection{Organization}

This paper is organized as follows. In Section~\ref{sec:background}, we provide background concerning bases for quasisymmetric functions and recall some known results that we will extend. In Section~\ref{sec:revdis}, we define the \emph{reverse dual immaculate quasisymmetric functions}, the \emph{reverse dual immaculate slide polynomials}, and important concepts we will use throughout the paper such as the \emph{weak descent composition}.  
In Section~\ref{sec:revtofs}, we prove Theorem~\ref{thm:rdistoslide}. 
In Section~\ref{sec:revstable}, we establish when a reverse dual immaculate slide polynomial is quasisymmetric, prove Theorem~\ref{thm:stablelimit}, and use this to prove that the formula of Theorem~\ref{thm:rdistoslide} stabilizes to the formula defining reverse dual immaculate quasisymmetric functions in terms of their fundamental expansion.  
In Section~\ref{sec:qktofs}, we establish a new formula for the fundamental slide expansion of a quasi-key polynomial. Finally, in Section~\ref{sec:distoyqk} we introduce the \emph{(Young) dual immaculate slide polynomials}, and Young versions of the quasi-key and fundamental slide bases. We extend an insertion algorithm of \cite{AHM18} to weak composition diagrams, and we use this algorithm as well as our standard formula from Section~\ref{sec:qktofs} to prove Theorem~\ref{thm:distoyqk}. As a corollary, the reverse dual immaculate slide polynomials expand positively in the quasi-key basis and therefore the basis of Demazure atoms; as a further corollary, the reverse dual immaculate quasisymmetric functions also expand positively in Demazure atoms.

In this paper, we introduce several families of combinatorial objects that generate the various bases we consider. To aid the reader in keeping track of these bases and the associated objects, we adopt some standardized notation. Throughout, we will use the word \emph{tableaux} to refer to assignments of integers to diagrams of \emph{compositions}; these generate families of \emph{quasisymmetric functions}. In contrast, we will use the word \emph{fillings} to refer to assignments of integers to diagrams of \emph{weak compositions}; these generate families of \emph{polynomials}. We routinely consider both semistandard and standard versions of both tableaux and fillings; we will always use $T$ to denote a semistandard object and $S$ to denote a standard object. Important bases and the associated tableaux/fillings we consider are summarized in the Appendix in Figures~\ref{table:qsym} (for quasisymmetric functions) and~\ref{fig:polytable} (for polynomials).

%%%%%%%%%%%%%%%%%%%%%%%%%%%%%%%%%%%%%%%%%%%
\section{Background on quasisymmetric functions}\label{sec:background}
%%%%%%%%%%%%%%%%%%%%%%%%%%%%%%%%%%%%%%%%%%%

In this section we review known bases for $\QSym$ relevant to this paper, including the fundamental quasisymmetric functions, the dual immaculate quasisymmetric functions, and the (Young) quasisymmetric Schur functions.  We also review the expansions of the dual immaculate quasisymmetric functions that we will extend to polynomials: the fundamental expansion~\cite{BBSSZ14} and the Young quasisymmetric Schur expansion~\cite{AHM18}.

A \emph{composition} $\alpha$ is a finite sequence $\alpha=(\alpha_1, \ldots , \alpha_\ell)$ of positive integers. The number $\ell$ is the \emph{length} of $\alpha$, sometimes denoted $\ell(\alpha)$, and the entries of $\alpha$ are called the \emph{parts} of $\alpha$. A \emph{weak composition} $a=(a_1, \ldots , a_\ell)$ is a finite sequence of nonnegative integers; the length of $a$ is $\ell$. When the parts of a composition $\alpha$ or weak composition $a$ sum to $n$, we say that $\alpha$ (respectively, $a$) is a composition (respectively, weak composition) of $n$.

The algebra $\Sym$ of symmetric functions is the collection of all bounded degree formal power series $f$ on an infinite alphabet $x_1,x_2, \hdots$ such that the coefficient of any monomial $x_{1}^{\alpha_1} x_{2}^{\alpha_2} \cdots x_{\ell}^{\alpha_\ell}$ in $f$ is equal to the coefficient of $x_{\sigma_1}^{\alpha_1} x_{\sigma_2}^{\alpha_2} \cdots x_{\sigma_\ell}^{\alpha_\ell}$ for any permutation $\sigma$.  Useful and widely-studied bases for $\Sym$ include the \emph{monomial symmetric functions}, the \emph{power sum symmetric functions}, the \emph{elementary symmetric functions}, the \emph{complete homogeneous symmetric functions}, and the \emph{Schur functions}; see, e.g., Chapter 7 of \cite{Sta99}.  The algebra $\Sym$ is contained inside a larger algebra of \emph{quasisymmetric functions}, denoted by $\QSym$, which is the collection of all bounded degree formal power series $f$ on an infinite alphabet $x_1, x_2, \hdots$ such that the coefficient of $x_{1}^{\alpha_1} x_{2}^{\alpha_2} \cdots x_{\ell}^{\alpha_\ell}$ in $f$ is equal to coefficient of $x_{j_1}^{\alpha_1} x_{j_2}^{\alpha_2} \cdots x_{j_\ell}^{\alpha_\ell}$ in $f$ for any sequence of positive integers $1 \le j_1 < j_2 < \cdots < j_\ell$ and any composition $(\alpha_1,\alpha_2, \hdots , \alpha_\ell)$.  See~\cite{Mas19} for an exposition of recent developments in the study of quasisymmetric functions.  For our purposes, we will sometimes restrict to symmetric and quasisymmetric functions in finitely many variables.

A natural basis for $\QSym$ is a generalization of the monomial symmetric functions called the \emph{monomial quasisymmetric functions} \cite{Ges84}. Given a composition $\alpha$ of length $\ell$, the monomial quasisymmetric function $M_\alpha$ is defined by 
\[M_{\alpha} (x_1, x_2, \hdots ) = \sum_{i_1 < i_2 < \cdots < i_\ell} x_{i_1}^{\alpha_1} x_{i_2}^{\alpha_2} \cdots x_{i_\ell}^{\alpha_\ell}.\]  
Another important and widely-used basis for $\QSym$ is the \emph{fundamental quasisymmetric functions} \cite{Ges84}. These are defined by 
\[F_{\alpha} = \sum_{\beta \preceq \alpha} M_{\beta},\] 
where $\beta \preceq \alpha$ if $\alpha$ can be obtained by summing consecutive entries of $\beta$.

The algebra $\Sym$ is also realised as a quotient of the algebra $\NSym$ of noncommutative symmetric functions~\cite{GKLLRT95}, which is generated by elements $\{H_1, H_2, \ldots\}$ with no relations. $\NSym$ has an (additive) basis consisting of the \emph{complete homogeneous functions} $H_\alpha$, where $H_\alpha$ is defined to be the product $H_{\alpha_1} \cdots H_{\alpha_\ell}$.    
It is dual to $\QSym$ under the pairing $\langle \cdot , \cdot \rangle$ defined by $\langle H_\alpha, M_\beta \rangle = \delta_{\alpha, \beta}$. 

The \emph{immaculate basis} $\{\im_\alpha\}$ for $\NSym$ was introduced in \cite{BBSSZ14} as a noncommutative analogue of the Schur basis for $\Sym$. The immaculate functions $\im_\alpha$ are constructed using a noncommutative analogue of Bernstein's creation operators, which generate Schur functions. 
To be precise, the \emph{noncommutative Bernstein operators} $\mathbb{B}_m$ generalize the classical Bernstein operators by replacing the complete homogeneous symmetric function with its noncommutative analogue $H_{\alpha}$ and the elementary symmetric function with Gessel's fundamental quasisymmetric function $F_{\alpha}$.  This substitution results in the operator $$\mathbb{B}_m = \sum_{i \ge 0} (-1)^i H_{m+i} F_{1^i}^{\perp},$$ where $F_{\alpha}^{\perp}$ is the linear transformation of $\NSym$ that is adjoint to multiplication by $F_{\alpha}$ in $\QSym$.  The \emph{immaculate function} $\im_\alpha$ is then defined by 
\[\im_{\alpha} := \mathbb{B}_{\alpha_1} \mathbb{B}_{\alpha_2} \cdots \mathbb{B}_{\alpha_m} (1).\] 
The immaculate basis is \emph{Schur-like} in the following precise sense. When the index is a partition $\lambda$, the immaculate function $\im_{\lambda}$ maps to the Schur function $s_{\lambda}$ under the natural projection from $\NSym$ to $\Sym$. (This projection is sometimes referred to as the \emph{forgetful map}, as it ``forgets" that the variables do not commute.)

The \emph{dual immaculate} basis $\{\di_\alpha\}$ for $\QSym$ is the dual basis to the immaculate basis for $\NSym$.  The \emph{dual immaculate quasisymmetric functions} $\di_\alpha$ are described combinatorially in terms of \emph{immaculate tableaux} \cite{BBSSZ14} as follows. Let $D(\alpha)$ denote the diagram of $\alpha$, i.e. the box diagram whose $i^{th}$ row from the bottom consists of $\alpha_i$ left-justified boxes. 

\begin{definition}\label{def:SSIT}
Let $\alpha$ be a composition of $n$. A \emph{semistandard immaculate tableau} of shape $\alpha$ is a filling of $D(\alpha)$ with positive integers so that:
\begin{itemize}
\item the sequence of entries in each row is weakly increasing from left to right;
\item the sequence of entries in the leftmost column is strictly increasing from bottom to top.
\end{itemize}
A semistandard immaculate tableau $S\in \SSIT(\alpha)$ is said to be \emph{standard} if the entries of $S$ are $1,2,\ldots , n$ each used exactly once. Let $\SSIT(\alpha)$ (respectively, $\SIT(\alpha)$) denote the set of semistandard (respectively, standard) immaculate tableaux of shape $\alpha$.  
\end{definition}

The \emph{weight} $\wt(T)$ of $T\in \SSIT(\alpha)$ is the weak composition whose $i^{th}$ part is the number of occurrences of $i$ in $T$. Note that we are using French notation rather than the English notation used in~\cite{BBSSZ14} and therefore our semistandard immaculate tableaux can be obtained from those in~\cite{BBSSZ14} by a reflection across a horizontal line.

\begin{ex}\label{ex:SSIT122}
The thirteen semistandard immaculate tableaux of shape $\alpha=(1,2,2)$ with entries in the set $\{1,2,3,4\}$ are shown in Figure~\ref{fig:di122} below.

\begin{figure}[ht]
  \begin{center}
    \begin{displaymath}
      \begin{array}{c@{\hskip2\cellsize}c@{\hskip2\cellsize}c@{\hskip2\cellsize}c@{\hskip2\cellsize}c@{\hskip2\cellsize}c@{\hskip2\cellsize}c@{\hskip2\cellsize}c@{\hskip2\cellsize}c}
        \vline\tableau{  3 & 3 \\ 2 & 2 \\ 1 } &
        \vline\tableau{   3 & 4 \\ 2 & 2 \\ 1 } &
        \vline\tableau{   4 & 4 \\ 2 & 2 \\ 1 } &
        \tableau{ 3 & 3 \\ 2 &3 \\ 1 }  &
        \tableau{  3 & 4 \\ 2 & 3 \\ 1 } &
        \tableau{   4 & 4 \\ 2 & 3 \\ 1} &
        \tableau{ 3 & 3 \\ 2 & 4 \\ 1 } \\ \\
     \tableau{   3 & 4 \\ 2 & 4 \\ 1 } &
       \tableau{  4 & 4 \\ 2 & 4 \\ 1 } &
       \tableau{  4 & 4 \\ 3 & 3 \\ 1 } &
       \tableau{ 4 & 4 \\ 3 & 4 \\ 1} &
       \tableau{   4 & 4 \\ 3 & 3 \\ 2 } &
       \tableau{   4 & 4 \\ 3 & 4 \\ 2 }  & &
      \end{array}
    \end{displaymath}
  \end{center}
  \caption{The thirteen immaculate tableaux of shape $(1,2,2)$ with entries in the set $\{1,2,3,4\}$.}\label{fig:di122}
  \end{figure}
\end{ex}

\begin{theorem}\cite{BBSSZ14}\label{thm:ditomonomial}
Let $\alpha$ be a composition. Then
\[\di_\alpha = \sum_{T\in \SSIT(\alpha)}x^{\wt(T)},\] 
where $x^{\wt(T)}$ is the monomial in which the exponent of $x_i$ is the $i^{th}$ part of $\wt(T)$.
\end{theorem}

\begin{ex}{\label{ex:dI122}}
Let $\alpha=(1,2,2)$. By Example~\ref{ex:SSIT122}, we have 
\begin{align*}
\di_{\alpha} ( x_1,x_2,x_3,x_4) & = x^{1220}+x^{1211}+x^{1202} + x^{1130}+2x^{1121}+2x^{1112} \\ 
                                                         & + x^{1103}+x^{1022}+x^{1013}+x^{0122}+x^{0113},
                                                         \end{align*}
where, for brevity, we write $a_1a_2\cdots a_\ell$ for weak compositions $a=(a_1,a_2, \ldots , a_\ell)$ appearing as exponent vectors.
\end{ex}

The integer $i$ is said to be a \emph{descent} of a standard immaculate tableau $S$ if $i+1$ is in a row strictly higher than $i$ in $S$. Let $\des_{\SIT}(S)$ denote the set of all descents of $S$. 
If the descents of $S$ are $i_1 < i_2 < \ldots < i_j$, then the \emph{descent composition} of $S$ is the composition $(i_1, i_2-i_1, \ldots i_j-i_{j-1}, n-i_j)$.

\begin{ex}\label{ex:SIT23}
Let $\alpha=(2,3)$. The four $\SIT$s of shape $\alpha$, along with their descent compositions, are shown in Figure~\ref{fig:SIT23} below. 
\end{ex}

\begin{figure}[ht]
  \begin{center}
    \begin{displaymath}
      \begin{array}{c@{\hskip2\cellsize}c@{\hskip2\cellsize}c@{\hskip2\cellsize}c@{\hskip2\cellsize}c@{\hskip2\cellsize}c}
        \vline\tableau{  3 & 4 & 5 \\ 1 & 2 } &
        \vline\tableau{  2 & 4 & 5 \\ 1 & 3 } &
        \vline\tableau{  2 & 3 & 5 \\ 1 & 4 } &
        \vline\tableau{  2 & 3 & 4 \\ 1 & 5 } \\ \\
        (2,3) & (1,2,2) & (1,3,1) & (1,4)
      \end{array}
    \end{displaymath}
  \end{center}
  \caption{The four $\SIT$s for $\alpha=(2,3)$ and their descent compositions.}\label{fig:SIT23}
  \end{figure}

The expansion of dual immaculate quasisymmetric functions into the fundamental basis for $\QSym$ is nonnegative. The following formula for the expansion was established by \cite{BBSSZ14}.

\begin{proposition}{\cite{BBSSZ14}}\label{prop:ditof}
The dual immaculate quasisymmetric functions $\di_{\alpha}$ expand in the fundamental basis via the formula 
\[\di_{\alpha} = \sum_{S\in \SIT(\alpha)}F_{\des_{\SIT}(S)},\]
where $F_{\des_{\SIT}(S)}$ is the fundamental quasisymmetric function associated to the descent composition of $S$.
\end{proposition}

\begin{ex}\label{ex:di23}
By Example~\ref{ex:SIT23}, we have
\[\di_{(2,3)} = F_{(2,3)} + F_{(1,2,2)} + F_{(1,3,1)} + F_{(1,4)}.\]
\end{ex}

In \cite{AHM18} it is shown that dual immaculate quasisymmetric functions also expand positively in the \emph{Young quasisymmetric Schur functions} \cite{LMvWbook}, another quasisymmetric analogue of the Schur basis.  The Young quasisymmetric Schur functions can be defined combinatorially as generating functions for certain tableaux of composition shape. 
 
\begin{definition}\label{def:YCT}
Let $\alpha$ be a composition of $n$. A \emph{Young composition tableau} of shape $\alpha$ is a filling of the boxes of the diagram of $\alpha$ with positive integers so that:
\begin{itemize}
\item the leftmost column entries strictly increase from bottom to top;
\item the row entries weakly increase from left to right;
\item the entries satisfy the following \emph{Young triple rule}: For any three boxes where $b$ and $c$ are adjacent and $a$ is in the same column as $c$ but in a lower row, so that the configuration of the entries is $\tableau{b & c \\ \\ & a}$, if $a \ge b$ then $a>c$. 
\end{itemize}
A Young composition tableau is \emph{standard} if its entries are $1, \ldots , n$, each used once. Let $\YCT(\alpha)$ (respectively, $\SYCT(\alpha)$) denote the set of all Young composition tableaux (respectively, standard Young composition tableaux) of shape $\alpha$.
\end{definition}

The \emph{Young quasisymmetric Schur function} indexed by a composition $\alpha$ is given by the formula
\[\yqs_{\alpha} = \sum_{T \in \YCT(\alpha)} x^{\wt(T)}.\] 
 
The Young quasisymmetric Schur functions were introduced as a variation on the \emph{quasisymmetric Schur} basis of~\cite{HLMvW11a}.  The quasisymmetric Schur function $\qs_{\rev(\alpha)}$ in $\ell$ variables (where $\rev(\alpha)$ is the composition obtained by reversing the order of the parts of $\alpha$) can be obtained from $\yqs_{\alpha}(x_1, \ldots , x_\ell)$ by replacing each variable $x_i$ with $x_{\ell-i+1}$.  
This reversal procedure is part of a broader division of bases for $\QSym$ and the polynomial ring into two different collections. These parallel constructions and their connections will be explored in greater detail in a forthcoming paper.

As evidence of the contrast between the two seemingly similar constructions, note that dual immaculate quasisymmetric functions expand positively into Young quasisymmetric Schur functions, while their expansion into quasisymmetric Schur functions includes negative coefficients. 

Below we state a positive formula for the expansion in the Young quasisymmetric Schur basis, given by \cite{AHM18} in terms of tableau fillings called \emph{dual immaculate recording tableaux}, or DIRTs. The relevant definitions of DIRTs, row strips, and row strip shape can be found in \cite{AHM18}. Alternatively, in Definition~\ref{def:DIRF} we generalise all of these concepts to \emph{dual immaculate recording fillings} (DIRFs), which are defined for diagrams of weak compositions; DIRTs may be realized as special cases of DIRFs.

\begin{theorem}{\cite[Theorem 1.1]{AHM18}}\label{thm:ditoyqs}
The dual immaculate quasisymmetric functions decompose into Young quasisymmetric Schur functions in the following way:

\[\di_{\alpha} = \sum_{\beta} c_{\alpha,\beta} \yqs_{\beta}\]

where $c_{\alpha,\beta}$ is the number of DIRTs of shape $\beta$ with row strip shape $\rev(\alpha)$.  
\end{theorem}

%%%%%%%%%%%%%%%%%%%%%%%%%%%%%%%%%%%%%%%%%%%%%%%%%%%
\section{Reverse dual immaculate slide polynomials}{\label{sec:revdis}}
%%%%%%%%%%%%%%%%%%%%%%%%%%%%%%%%%%%%%%%%%%%%%%%%%%%

We begin by defining a `reverse' version, $\rdi$, of the dual immaculate quasisymmetric functions, analogous to the distinction between Young quasisymmetric Schur and quasisymmetric Schur functions.  
We then introduce the \emph{reverse dual immaculate slide polynomials}, $\rdis$, our first lifting of the dual immaculate quasisymmetric functions to the polynomial ring, and prove they form a basis for polynomials. Finally, we expand the notion of a \emph{weak descent composition} from \cite{Assaf:nonsymmetric} to our setting and prove several properties which will be useful later. For our current purposes, it is more natural to work with the reverse versions rather than the Young versions. In particular, many of the combinatorial bases for the ring of polynomials are naturally defined in terms of reverse fillings, i.e., where entries decrease along rows.

\subsection{Reverse immaculate tableaux}

We now introduce the objects, reverse immaculate tableaux, needed to define the reverse dual immaculate quasisymmetric functions.

\begin{definition}\label{def:SSRIT}
Let $\alpha$ be a composition of $n$. A \emph{semistandard reverse immaculate tableau} of shape $\alpha$ is a filling of the boxes of the diagram of $\alpha$ with positive integers so that:
\begin{itemize}
\item the sequence of entries in each row is weakly \emph{decreasing} from left to right;
\item the sequence of entries in the leftmost column is strictly increasing from bottom to top.
\end{itemize}
A semistandard reverse immaculate tableau of shape $\alpha$ is said to be \emph{standard} if its entries are $1,2,\ldots , n$ each used exactly once. Let $\SSRIT(\alpha)$ (respectively, $\SRIT(\alpha)$) denote the set of semistandard (respectively, standard) reverse immaculate tableaux of shape $\alpha$. 
\end{definition}

As for standard immaculate tableaux, there is a notion of descent sets for standard reverse immaculate tableaux.  Throughout the paper, whenever we refer to the \emph{descent composition}, we construct it from the descent set via the standard correspondence between subsets of $[ n-1 ]$ and compositions of $n$ as described in the following definition.

\begin{definition}{\label{def:desSRIT}}
The descent set, $\des_{\SRIT}(S)$, of a standard reverse immaculate tableau $S$ is the set of all $i$ such that $i+1$ is in a row strictly higher than $i$ in $S$. If $i \in \des_{\SRIT}(S)$, then $i$ is said to be a \emph{descent} of $S$.  The \emph{descent composition}  is the composition corresponding to the set of descents of $S$, i.e., if the descents of $S$ are $i_1 < i_2, \hdots < i_j$ then the corresponding descent composition is $(i_1, i_2-i_1, i_3-i_2, \hdots, n-i_j)$.  
\end{definition}

We note the definition of descent for standard reverse immaculate tableaux is identical to that for standard immaculate tableaux.

\begin{ex}\label{ex:SRIT32}
Let $\alpha=(3,2)$. The four $\SRIT$s of shape $\alpha$, along with their descent compositions, are shown in Figure~\ref{fig:SRIT32} below. 
\end{ex}

\begin{figure}[ht]
  \begin{center}
    \begin{displaymath}
      \begin{array}{c@{\hskip2\cellsize}c@{\hskip2\cellsize}c@{\hskip2\cellsize}c@{\hskip2\cellsize}c@{\hskip2\cellsize}c}
        \vline\tableau{  5 & 4  \\ 3 & 2 & 1 } &
        \vline\tableau{  5 & 3  \\ 4 & 2 & 1 } &
        \vline\tableau{  5 & 2  \\ 4 & 3 & 1 } &
        \vline\tableau{  5 & 1  \\ 4 & 3 & 2 } \\ \\
        (3,2) & (2,2,1) & (1,3,1) & (4,1)
      \end{array}
    \end{displaymath}
  \end{center}
  \caption{The 4 $\SRIT$s for $\alpha=(3,2)$ and their descent compositions.}\label{fig:SRIT32}
  \end{figure}

We define the reverse dual immaculate quasisymmetric functions $\rdi_{\alpha}$ to be the sum of the fundamental quasisymmetric functions indexed by the descent compositions of the elements of $\SRIT(\alpha)$, i.e., 
\begin{equation}\label{eq:direv}
\rdi_{\alpha} = \sum_{S\in \SRIT(\alpha)}F_{\des_{\SRIT}(S)}.
\end{equation}

\begin{ex}\label{ex:direv32}
By Example~\ref{ex:SRIT32}, depicted in Figure~\ref{fig:SRIT32}, we have
\[\rdi_{(3,2)} = F_{(3,2)} + F_{(2,2,1)} + F_{(1,3,1)} + F_{(4,1)}.\]
\end{ex}

The dual immaculate and reverse dual immaculate quasisymmetric functions are related via a reversal of the variable set and the composition.

\begin{lemma}\label{lem:SITSRITbij}
For any composition $\alpha$, there is a descent composition reversing bijection $\theta$ between $\SRIT(\rev(\alpha))$ and $\SIT(\alpha)$. 
\end{lemma}
\begin{proof}
Let $\alpha$ be a composition of $n$ and $S\in \SRIT(\rev(\alpha))$. Let $\theta(S)\in \SIT(\alpha)$ be given by exchanging each entry $i$ in $S$ with $n+1-i$ and then reversing the order of the rows of $S$. This map $\theta$ is clearly an involution from $\SRIT(\rev(\alpha))$ to $\SIT(\alpha)$, and hence sends each element of $\SRIT(\rev(\alpha))$ to a unique element of $\SIT(\alpha)$. It is straightforward to see that $i+1$ is strictly above $i$ in $S$ if and only if $n-i+1$ is strictly above $n-i$ in $\theta(S)$, hence $\theta$ reverses the descent composition. 
\end{proof}

Note that, for example, each $\SRIT$ in Figure~\ref{fig:SRIT32} is sent under $\theta$ to the corresponding $\SIT$ in Figure~\ref{fig:SIT23}.  Now compare the fundamental expansion of $\rdi_{(3,2)}$ in Example~\ref{ex:direv32} and the fundamental expansion of $\di_{(2,3)}$ in Example~\ref{ex:di23}. 

We therefore obtain the following proposition.
\begin{proposition}\label{prop:ditordi}
For any composition $\alpha$ and any positive integer $\ell$, we have \[\di_{\alpha}(x_1, \ldots , x_\ell) = \rdi_{\rev(\alpha)}(x_\ell , \ldots , x_1).\]
\end{proposition}
\begin{proof}
The left-hand side is equal to $\sum_{S\in \SIT(\alpha)}F_{\des_{\SIT}(S)}(x_1,\ldots , x_\ell)$ by Proposition~\ref{prop:ditof}, while the right-hand side is equal to $\sum_{S\in \SRIT(\rev(\alpha))}F_{\des_{\SRIT}(S)}(x_\ell, \ldots , x_1)$ by definition.  The proposition then follows from Lemma~\ref{lem:SITSRITbij} and the fact that $F_{\beta}(x_1, \ldots, x_\ell) = F_{\rev(\beta)}(x_\ell, \ldots , x_1)$ for any positive integer $\ell$ and any composition $\beta$.
\end{proof}

We now define a polynomial-ring analogue of the reverse dual immaculate quasiymmetric functions. The \emph{reverse} of a weak composition $a$, denoted $\rev(a)$, is the weak composition of the same length as $a$ formed by writing the parts (including $0$ parts) of $a$ in reverse order, e.g. $\rev(0,3,0,2) = (2,0,3,0)$. Given a weak composition $a$ of $n$ of length $\ell$, the \emph{diagram} of $a$, denoted $D(a)$, consists of $\ell$ left-justified rows such that the $i^{th}$ row from the bottom contains $a_i$ boxes.  

\begin{definition}{\label{def:SSRIF}}
Let $a$ be a weak composition of $n$. We define a \emph{semistandard reverse immaculate filling} (SSRIF) of shape $a$ to be a filling of the boxes in the diagram of $a$ with positive integers satisfying the following properties.
\begin{enumerate}[leftmargin=1.8cm] 
\item[(SSRIF1)] Row entries weakly decrease from left to right.
\item[(SSRIF2)] Entries in the leftmost column strictly increase from bottom to top.
\item[(SSRIF3)] Entries in the $i^{th}$ row from the bottom do not exceed $i$.
\end{enumerate}
A \emph{standard reverse immaculate filling} of shape $a$ is a filling of $D(a)$ with $1,\ldots , n$ (each used once) that satisfies (SSRIF1) and (SSRIF2), but not necessarily (SSRIF3). Denote the set of semistandard (respectively, standard) reverse immaculate fillings of shape $a$ by $\SSRIF(a)$ (respectively, $\SRIF(a)$). 
\end{definition}

\begin{remark}
Note that $\SRIF(a)$ is not in general a subset of $\SSRIF(a)$, since elements of $\SRIF(a)$ are not required to satisfy (SSRIF3). This situation, in which standard objects do not need to satisfy a row index condition imposed on semistandard objects, will be a common thread in this paper. We note this is unlike the situation typically encountered in the literature, in which standard objects are a subset of semistandard objects; for example, $\SIT(\alpha)$ is a subset of $\SSIT(\alpha)$ (Definition~\ref{def:SSIT}). 
\end{remark}

\begin{ex}
Figure~\ref{fig:SSRIF} depicts the $30$ $\SSRIF$s of shape $a=(0,3,0,2)$. 
\end{ex}

\begin{figure}[ht]
  \begin{center}
    \begin{displaymath}
      \begin{array}{c@{\hskip2\cellsize}c@{\hskip2\cellsize}c@{\hskip2\cellsize}c@{\hskip2\cellsize}c@{\hskip2\cellsize}c}
        \vline\tableau{  4 & 4  \\ \\  2 & 2 & 2 \\ \\\hline } &
        \vline\tableau{  4 & 4  \\ \\  2 & 2 & 1 \\ \\\hline } &
        \vline\tableau{  4 & 4  \\ \\  2 & 1 & 1 \\ \\\hline } &
        \vline\tableau{  4 & 4  \\ \\  1 & 1 & 1 \\ \\\hline } &
        \vline\tableau{  4 & 3  \\ \\  2 & 2 & 2 \\ \\\hline } &
        \vline\tableau{  4 & 3  \\ \\  2 & 2 & 1 \\ \\\hline } \\ \\
        \vline\tableau{  4 & 3  \\ \\  2 & 1 & 1 \\ \\\hline } &
        \vline\tableau{  4 & 3  \\ \\  1 & 1 & 1 \\ \\\hline } &
        \vline\tableau{  4 & 2  \\ \\  2 & 2 & 2 \\ \\\hline } &
        \vline\tableau{  4 & 2  \\ \\  2 & 2 & 1 \\ \\\hline } &
        \vline\tableau{  4 & 2  \\ \\  2 & 1 & 1 \\ \\\hline } &
        \vline\tableau{  4 & 2  \\ \\  1 & 1 & 1 \\ \\\hline } \\ \\
        \vline\tableau{  4 & 1  \\ \\  2 & 2 & 2 \\ \\\hline } &
        \vline\tableau{  4 & 1  \\ \\  2 & 2 & 1 \\ \\\hline } &
        \vline\tableau{  4 & 1  \\ \\  2 & 1 & 1 \\ \\\hline } &
        \vline\tableau{  4 & 1  \\ \\  1 & 1 & 1 \\ \\\hline } &
        \vline\tableau{  3 & 3  \\ \\  2 & 2 & 2 \\ \\\hline } &
        \vline\tableau{  3 & 3  \\ \\  2 & 2 & 1 \\ \\\hline } \\ \\
        \vline\tableau{  3 & 3  \\ \\  2 & 1 & 1 \\ \\\hline } &
        \vline\tableau{  3 & 3  \\ \\  1 & 1 & 1 \\ \\\hline } &
        \vline\tableau{  3 & 2  \\ \\  2 & 2 & 2 \\ \\\hline } &
        \vline\tableau{  3 & 2  \\ \\  2 & 2 & 1 \\ \\\hline } &
        \vline\tableau{  3 & 2  \\ \\  2 & 1 & 1 \\ \\\hline } &
        \vline\tableau{  3 & 2  \\ \\  1 & 1 & 1 \\ \\\hline } \\ \\
        \vline\tableau{  3 & 1  \\ \\  2 & 2 & 2 \\ \\\hline } &
        \vline\tableau{  3 & 1  \\ \\  2 & 2 & 1 \\ \\\hline } &
        \vline\tableau{  3 & 1  \\ \\  2 & 1 & 1 \\ \\\hline } &
        \vline\tableau{  3 & 1  \\ \\  1 & 1 & 1 \\ \\\hline } &        
        \vline\tableau{  2 & 2  \\ \\  1 & 1 & 1 \\ \\\hline } &
        \vline\tableau{  2 & 1  \\ \\  1 & 1 & 1 \\ \\\hline } 
      \end{array}
    \end{displaymath}
  \end{center}
  \caption{The 30 $\SSRIF$s for $a=(0,3,0,2)$.}\label{fig:SSRIF}
  \end{figure}

\begin{ex}
Figure~\ref{fig:0302slide} depicts the $4$ $\SRIF$s of shape $a=(0,3,0,2)$.
\end{ex}  
  
\begin{figure}[ht]
  \begin{center}
    \begin{displaymath}
      \begin{array}{c@{\hskip2\cellsize}c@{\hskip2\cellsize}c@{\hskip2\cellsize}c}
        \vline\tableau{  5 & 4  \\ \\  3 & 2 & 1 \\ \\\hline } &
        \vline\tableau{  5 & 3  \\ \\  4 & 2 & 1 \\ \\\hline } &
        \vline\tableau{  5 & 2  \\ \\  4 & 3 & 1 \\ \\\hline } &
        \vline\tableau{  5 & 1  \\ \\  4 & 3 & 2 \\ \\\hline } \\         
             \end{array}
    \end{displaymath}
  \end{center}
  \caption{The 4 $\SRIF$s for $a=(0,3,0,2)$.}\label{fig:0302slide}
  \end{figure}

Given a weak composition $a$, define the \emph{reverse dual immaculate slide polynomial} $\rdis_{a}$ by $$\rdis_a = \sum_{T \in \SSRIF(a)} x^{\wt(T)}.$$

\begin{ex}
Let $a=(0,3,0,2)$. Then, as computed by Figure~\ref{fig:SSRIF}, we have 

\begin{align*}\rdis_a & = x^{0302}+x^{1202}+x^{2102}+x^{3002} +x^{0311} +x^{1211}+x^{2111}+x^{3011} +x^{0401} +x^{1301} \\ 
                                & +x^{2201}+x^{3101} +x^{1301} +x^{2201}+x^{3101}+x^{4001} +x^{0320} +x^{1220}+x^{2120}+x^{3020} \\ 
                                & +x^{0410} +x^{1310}+x^{2210}+x^{3110} +x^{1310} +x^{2210}+x^{3110}+x^{4010} +x^{3200} +x^{4100}.
\end{align*}
\end{ex}

\begin{theorem}
The set 
\[\{\rdis_a : a \mbox{ is a weak composition of length } \ell\}\]
of reverse dual immaculate slide polynomials forms a basis for polynomials in $\ell$ variables.
\end{theorem}
\begin{proof}
Since bases for polynomials in $\ell$ variables are indexed by weak compositions of length $\ell$, it is enough to show that reverse dual immaculate slide polynomials are a spanning set. To do this, consider the lexicographic order $>_{{\rm lex}}$ on weak compositions, where $a >_{{\rm lex}} b$ if $a_i>b_i$ for the smallest index $i$ such that $a_i\neq b_i$. It is immediate from (SSRIF3) that $x^a$ is the monomial in $\rdis_a$ whose exponent vector is minimal in $>_{{\rm lex}}$.  

Let $a$ be a weak composition of $n$ of length $\ell$. Then $x^a-\rdis_a$ is a sum of monomials whose exponent vectors are strictly greater in lexicographic order than $a$, and moreover all are weak compositions of $n$, since reverse dual immaculate slide polynomials are homogeneous. Now, supposing the monomial in $x^a-\rdis_a$ with minimal exponent vector is $c_bx^b$, subtract $c_b\rdis_b$ from $x^a-\rdis_a$.  The exponent vector of each monomial in this expression is strictly greater in lexicographic order than $b$, and is a weak composition of $n$. Continue this procedure until it terminates; termination is guaranteed in a finite number of steps since there are only finitely many weak compositions of $n$ of length $\ell$ and the minimal term at each stage is strictly greater in lexicographic order than the minimal term at the previous stage. Hence we obtain a linear combination of $x^a$ and reverse dual immaculate slide polynomials that is equal to zero. It follows that every monomial $x^a$ is a linear combination of reverse dual immaculate slide polynomials.
\end{proof}

The \emph{reading word} of a semistandard or standard reverse immaculate filling is the word obtained by reading the entries along rows from right to left, starting at the top row and proceeding to the bottom row.  This order on the entries is called the \emph{reading order}; see Example~\ref{ex:SSRIF}.  The \emph{descent set} of a standard reverse immaculate filling $S$, denoted by $\des_{\SRIF}(S)$, is the set of all $i$ such that $i+1$ is in a strictly higher row, and the \emph{descent composition} of $S$ is the composition associated to the descent set (as in Definition~\ref{def:desSRIT}).

\begin{remark}
Note that this reading word is subtly different from that of an immaculate tableau (or immaculate filling), which is given by reading the entries along rows from left to right, top to bottom (see, e.g., \cite{AHM18}). The descent set, however, is the same.  
\end{remark}

Let $a$ be a weak composition of $n$. For each $T\in \SSRIF(a)$ there is a unique \emph{standardization} $\std_{\SSRIF}(T)\in \SRIF(a)$ of $T$.  This standardization is obtained by replacing the $i^{th}$ smallest entry by $i$ for all $1 \le i \le n$, where if two entries are identical then the entry appearing first in reading order is considered to be smaller. Since standardization preserves the relative order between entries, it is immediate that (SSRIF1), (SSRIF2) and (SSRIF4) are preserved, hence $\std_{\SSRIF}(T)\in \SRIF(a)$.

\begin{ex}\label{ex:SSRIF}
Let $a=(0,2,0,4,3)$ be a weak composition of $n=9$ of length $\ell=5$. The given filling $T$ is a semistandard reverse immaculate filling of shape $a$ with reading word $2 \; 4 \; 5 \; 1\; 2 \; 3 \; 3 \; 1 \; 2$ and standardization $S=\std_{\SSRIF}(T)$.  

\[T \mbox{ $=$ } \vline\tableau{  5 & 4 & 2 \\ 3 & 3 & 2 & 1 \\ \\  2 & 1 \\ \\\hline } \qquad S=\std_{\SSRIF}(T) \mbox{ $=$ }   \vline\tableau{  9 & 8 & 3 \\ 7 & 6 & 4 & 1 \\ \\  5 & 2 \\ \\\hline}\]
We have $\des_{\SRIF}(S) = \{ 2,5,7 \}$; hence the descent composition of $S$ is $(2,3,2,2)$.
\end{ex}

\subsection{Weak descent compositions} 
A goal is to derive a positive formula for the \emph{fundamental slide} expansion of a reverse dual immaculate slide polynomial, thus giving a polynomial lifting of the formula (\ref{eq:direv}), i.e., of the reverse version of the fundamental expansion of a dual immaculate quasisymmetric function \cite{BBSSZ14} (Proposition~\ref{prop:ditof}). To this end, we will associate a fundamental slide polynomial to each standard reverse immaculate filling. The fundamental slide polynomials, which we will define in Section~\ref{sec:revtofs}, were introduced in \cite{Assaf.Searles:1} as a lifting of the fundamental basis for quasisymmetric functions to a basis for polynomials.  
To associate fundamental slide polynomials to $\SRIF$s, we adapt the notion of a \emph{weak descent composition} from \cite[Definition 2.3]{Assaf:nonsymmetric} to our context. 

\begin{definition}\label{def:wdesrdis}
Let $a$ be a weak composition of $n$ of length $\ell$. The  \emph{weak descent composition} of $S\in \SRIF(a)$, denoted $\wdes_{\SRIF}(S)$, is the weak composition of length $\ell$ obtained as follows. First, decompose the word $n\; n-1 \ldots 2\; 1$ into runs by placing a bar between $i+1$ and $i$ whenever $i+1$ is strictly above $i$ in $S$; i.e., whenever $i\in \des_{\SRIF}(S)$. Suppose there are $k$ runs, i.e., the run decomposition yields $r_k | r_{k-1} | \ldots | r_1$ where each $r_j$ denotes a run. 

Define a strictly decreasing sequence of numbers $p_k, \ldots , p_1$ recursively as follows. Let $p_k$ be the row index of the box containing $n$.  
Now, let $j<k$ and suppose $p_k, \ldots , p_{j+1}$ are known. Define $p_j$ to be the smaller of $p_{j+1}-1$ and the index of the lowest row containing an entry from the $j^{th}$ run $r_j$.

If any $p_j$ is nonpositive, then $\wdes_{\SRIF}(S)=\emptyset$. Otherwise, $\wdes_{\SRIF}(S)$ is the weak composition of length $\ell$ whose $p_j^{th}$ part is the number of entries in the run $r_j$, and whose other parts are all zero. 
\end{definition}

\begin{remark}\label{rmk:wdessemistandard}
Definition~\ref{def:wdesrdis} extends straightforwardly to give runs $r_j$, numbers $p_j$, and thus a weak descent composition for a semistandard reverse immaculate filling, which will be used in the proof of Theorem~\ref{thm:rdistoslide}. 
To find $\wdes_{\SSRIF}(T)$ for $T\in \SSRIF(a)$, write the entries of $T$ in order from largest to smallest, where if two entries are identical the one earlier in reading order is considered to be smaller. This forms a word $w_n \; w_{n-1} \ldots w_2 \; w_1$, where $w_n$ is the largest entry in $T$. Decompose this word into runs $r_j$ by placing a bar between the letters $w_{i+1}$ and $w_i$ whenever $w_{i+1}$ is strictly above $w_i$ in $T$. Then follow the same procedure as for $\wdes_{\SRIF}(S)$ to find the numbers $p_j$, which then define the weak composition $\wdes_{\SSRIF}(T)$.
\end{remark}

\begin{ex}\label{ex:wdes}
We compute $\wdes_{\SRIF}(S)$ for $S$ from Example~\ref{ex:SSRIF}. We have $\des_{\SRIF}(S) = \{2,5,7\}$ giving the run decomposition $98|76|543|21 = r_4|r_3|r_2|r_1$. The box containing $9$ is in row $5$, hence $p_4=5$. Now $p_3$ is the smaller of $p_4-1=4$ and the lowest row index containing $7$ or $6$. Those entries both appear in row $4$, hence $p_3=4$. Next, $p_2$ is the smaller of $p_3-1 = 3$ and the index of the lowest row containing $5, 4$ or $3$, which is row $2$. Hence $p_2=2$. Finally $p_1$ is the smaller of $p_2-1=1$ and the index of the lowest row containing $2$ or $1$, which is row $2$. Hence $p_1=1$. As a result, $\wdes_{\SRIF}(S) = (2,3,0,2,2)$.

We can also find $\wdes_{\SSRIF}(T)$  for $T$ from Example~\ref{ex:SSRIF}. The run decomposition is $54|33|222|11$. Applying the procedure as for $S$ above, we find that  $\wdes_{\SSRIF}(T) = (2,3,0,2,2)$ also.
\end{ex}

\begin{ex}\label{ex:notwdes}
Consider
\[S \mbox{ $=$ } \vline\tableau{  5 & 2 \\ \\  4 & 3 & 1 \\\hline} \in \SRIF(3,0,2).\] 
We have $\des_{\SRIF}(S) = \{1,4\}$, hence we have the run decomposition $5|432|1 = r_3|r_2|r_1$. The box containing $5$ is in row $3$, hence $p_3=3$.
 
Now $p_2$ is the smaller of $p_3-1=2$ and the lowest row index containing $4,3$ or $2$, which is row $1$. Hence $p_2=1$. But since $p_{j+1}>p_j$ for each $j$ by definition, we must have $p_1\le 0$, thus $\wdes_{\SRIF}(S)=\emptyset$. 
\end{ex}

Note that the weak composition $\wdes_{\SRIF}(S)$ is empty precisely when there is a run containing entries in a lower row of $S$ than its index.  That is, $\wdes_{\SRIF}(S) = \emptyset$ if and only if there exists a $j$ such that the $j^{th}$ run includes entries in a row below row $j$.

\begin{lemma}\label{lem:wdespreservedstd}
Standardization preserves weak descent compositions; that is, if $\std_{\SSRIF}(T)=S$ then $\wdes_{\SSRIF}(T)=\wdes_{\SRIF}(S)$. 
\end{lemma}
\begin{proof}
Let $w_n, \ldots , w_1$ be the entries of $T$ written in weakly decreasing order (if two entries are indentical, the one earlier in reading order is considered smaller). Recall (Remark~\ref{rmk:wdessemistandard}) that the run decomposition of $T$ is constructed by placing a bar between $w_{i+1}$ and $w_{i}$ whenever $w_{i +1}$ is strictly above $w_{i}$ in $T$. Applying the standardization procedure replaces $w_i$ in $T$ with $i$ for each $1\le i \le n$. This preserves the positions of the bars (in the run decomposition of $S$) and thus the weak descent composition since the position of $w_i$ in $T$ is exactly the position of $i$ in $S$.
\end{proof}

%%%%%%%%%%%%%%%%%%%%%%%%%%%%%%%%%%%%%%%%%%%%%%%%%%%%%%%%%%%%%%%%%
\section{Reverse dual immaculate slide polynomials expand positively into fundamental slide polynomials}{\label{sec:revtofs}
%%%%%%%%%%%%%%%%%%%%%%%%%%%%%%%%%%%%%%%%%%%%%%%%%%%%%%%%%%%%%%%%%

In this section, we prove Theorem~\ref{thm:rdistoslide}, providing a positive formula for the expansion of reverse dual immaculate slide polynomials into the fundamental slide basis of \cite{Assaf.Searles:1}.  

First we recall the fundamental semistandard skyline fillings introduced in~\cite{Sea}. 

\begin{definition}\label{def:FSSF}
Given a weak composition $a$, a \emph{fundamental semistandard skyline filling} of shape $a$ is a filling of the boxes of $D(a)$ with positive integers satisfying the following properties.
\begin{enumerate}
\item[(F1)] Row entries weakly decrease from left to right.
\item[(F2)] If a box $\mathfrak{b}$ is in a strictly higher row than a box $\mathfrak{b'}$, then the entry in $\mathfrak{b}$ is strictly greater than the entry in $\mathfrak{b'}$.
\item[(F3)] Entries in the $i^{th}$ row from the bottom do not exceed $i$. 
\end{enumerate}
Let $\FSSF(a)$ denote the set of all fundamental semistandard skyline fillings of shape $a$.
\end{definition}
\begin{ex}\label{ex:fs021}
Figure~\ref{fig:fs021} depicts the $\FSSF$s of shape $a=(0,2,1)$ 
\end{ex}

\begin{figure}[ht]
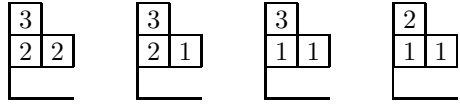

  \begin{center}
    \begin{displaymath}
      \begin{array}{c@{\hskip2\cellsize}c@{\hskip2\cellsize}c@{\hskip2\cellsize}c}
        \vline\tableau{   3 \\ 2 & 2 \\ \\\hline } &
        \vline\tableau{   3 \\ 2 & 1 \\ \\\hline } &
        \vline\tableau{   3 \\ 1 & 1 \\ \\\hline } &
        \vline\tableau{   2 \\ 1 & 1 \\ \\\hline } \\
             \end{array}
    \end{displaymath}
  \end{center}
  \caption{The 4 $\FSSF$s for $a=(0,2,1)$.}\label{fig:fs021}
\end{figure}

We may take the following result as definitional for the fundamental slide polynomials.

\begin{proposition}\cite{Sea}
Let $a$ be a weak composition. Then the fundamental slide polynomial $\fs_a$ is the generating function for the fundamental semistandard skyline fillings of shape $a$, that is,
\[\fs_a=\sum_{K\in \FSSF(a)}x^{\wt(K)}.\]
\end{proposition}

\begin{ex}
We have $\fs_{(0,2,1)} = x^{021} + x^{111} + x^{201} + x^{210}$, as computed by the $\FSSF$s from Figure~\ref{fig:fs021}.
\end{ex}

\begin{lemma}\label{lem:stdeq}
Let $a$ be a weak composition. Then standardization partitions $\SSRIF(a)$ into equivalence classes.  Specifically,
\[\SSRIF(a) = \bigsqcup_{S\in \SRIF(a)} \left\{T\in \SSRIF(a) : \std_{\SSRIF}(T) = S\right\}.\]
\end{lemma}
\begin{proof}
By the definition of standardization, given $T\in \SSRIF(a)$ there is a unique $S\in \SRIF(a)$ such that $\std_{\SSRIF}(T)=S$.
\end{proof}

\begin{lemma}\label{lem:bothsidesempty}
Let $a$ be a weak composition and $S\in \SRIF(a)$. Then $\wdes_{\SRIF}(S)=\emptyset$ if and only if there is no $T\in \SSRIF(a)$ that standardizes to $S$.
\end{lemma}
\begin{proof}
We first show that $\wdes_{\SSRIF}(T)$ is never equal to $\emptyset$ for any $T\in \SSRIF(a)$. It is enough to show that for the run decomposition $r_k | \cdots | r_1$ of $T$, we have $p_j\ge j$ for each $1\le j \le k$. Since for any $i$ all $i$'s appear in the same run, we must have all $1$'s in $T$ in $r_1$, all $2$'s in $r_1$ or $r_2$, etc. In particular $r_j$ cannot involve any entries smaller than $j$. By (SSRIF3), this implies all entries of run $r_j$ must appear in row $j$ or higher. This forces $p_j\ge j$ for all $j$, so $\wdes_{\SSRIF}(T)\neq \emptyset$. Therefore, if $T\in \SSRIF(a)$ standardizes to $S$, we have $\emptyset \neq \wdes_{\SSRIF}(T)$.  Since $\wdes_{\SSRIF}(T) = \wdes_{\SRIF}(S)$ by Lemma~\ref{lem:wdespreservedstd}, we have $\wdes_{\SRIF}(S)\neq \emptyset$.

Conversely, suppose $\wdes_{\SRIF}(S)\neq\emptyset$. Construct $T$ by replacing every entry of run $r_j$ of $S$ with the value $p_j$. Then if an entry $i$ is smaller than an entry $j$ in $S$, the value $i$ is replaced with is less than or equal to the value $j$ is replaced with, with equality if and only if $i$ and $j$ are in the same run. Therefore, since entries of $S$ decrease along rows, the entries of $T$ weakly decrease along rows, satisfying (SSRIF1). Moreover, since the leftmost column of $S$ is strictly increasing from bottom to top, no run may use more than one entry from the leftmost column of $S$.  This implies the leftmost column of $T$ is also strictly increasing from bottom to top, satisfying (SSRIF2). For (SSRIF3), by construction, the index of the lowest row that a $p_j$ is placed into is no smaller than $p_j$.

Finally, note the construction of $T$ involves replacing every element in a run in $S$ with a particular value, and the value assigned to each run strictly increases when going from $r_j$ to $r_{j+1}$. 
Hence standardizing $T$ involves replacing all $p_1$'s with $1,\ldots , |r_1|$ in reading order (where $|r_1|$ is the number of entries in the run $r_1$), all $p_2$'s with $|r_1|+1, \ldots , |r_1|+|r_2|$ in reading order, etc, which is exactly the inverse procedure to constructing $T$ from $S$. Hence $T$ standardizes to $S$. 
\end{proof}

We are now ready to prove Theorem~\ref{thm:rdistoslide}, which we recall states that for a weak composition $a$, 
\[\rdis_{a} = \sum_{S\in \SRIF(a)} \fs_{\wdes_{\SRIF}(S)}.\]

\noindent\emph{Proof of Theorem~\ref{thm:rdistoslide}}. 
We claim that for each $S\in \SRIF(a)$ there is a weight-preserving bijection
\begin{equation}\label{eq:SSRIFbijection}
\Psi_S: \left\{T\in \SSRIF(a) : \std_{\SSRIF}(T) = S\right\} \,\,\,\, \rightarrow \,\,\,\, \FSSF(\wdes_{\SRIF}(S)).
\end{equation} 
Since the right-hand side of (\ref{eq:SSRIFbijection}) generates $\fs_{\wdes_{\SRIF}(S)}$, the theorem will follow from this claim and Lemma~\ref{lem:stdeq}. 

To prove the claim, note that by Lemma~\ref{lem:bothsidesempty}, for any $S\in \SRIF(a)$ the left hand side of (\ref{eq:SSRIFbijection}) is empty if and only if the right hand side is empty. Hence we only need to consider those $S\in \SRIF(a)$ such that there is some $T\in \SSRIF(a)$ that standardizes to $S$. For those $T\in \SSRIF(a)$ that standardize to a given $S$, we define $\Psi_S(T)$ to be the filling obtained by placing the entries in run $r_j$ of $T$ into row $p_j$ of $\Psi(T)$ (in the order they appear in the run), where the runs $r_j$ and numbers $p_j$ are those obtained in computing the weak descent composition of $T$ (Remark~\ref{rmk:wdessemistandard}). See Example~\ref{ex:PsiSbijection} for an example of this map.

By construction, $\Psi_S$ is weight-preserving and $\Psi_S(T)$ has shape $\wdes_{\SSRIF}(T)$. Since the weak descent compositions of $T$ and $S$ are the same by Lemma~\ref{lem:wdespreservedstd}, we have that $\Psi_S(T)$ indeed has shape $\wdes_{\SRIF}(S)$. We claim that $\Psi_S(T)$ is an $\FSSF$. Since entries of each run of $T$ are placed into a row of $\Psi_S(T)$ in order from largest to smallest, the entries of $\Psi_S(T)$ decrease along rows, so (F1) is satisfied. Moreover, since for any $j$ all entries of run $r_{j+1}$ in $T$ are strictly larger than each entry of run $r_j$, we have that all entries in each row of $\Psi_S(T)$ are strictly larger than all entries in any lower row of $\Psi_S(T)$, so (F2) is satisfied. The fact that no entry in a row of $\Psi_S(T)$ exceeds its row index (and thus (F3) is satisfied) follows from the definition (Remark~\ref{rmk:wdessemistandard}) of the numbers $p_j$ for the weak descent composition of $T$. Hence the image of $\Psi_S$ is contained in $\FSSF(\wdes_{\SRIF}(S))$.

We must now show that for any $S\in \SRIF(a)$ and $K\in \FSSF(\wdes_{\SRIF}(S))$, there is a unique $T\in \SSRIF(a)$ such that $\std_{\SSRIF}(T)=S$ and $\wt(T) = \wt(K)$.

Given $S$ and $K$, we construct $T$ as follows. Let the runs $r_j$ and numbers $p_j$ be those obtained from $S$ in the construction of $\wdes_{\SRIF}(S)$ (Definition~\ref{def:wdesrdis}). Replace the entries of run $r_j$ in $S$, in order from largest to smallest, with the entries of row $p_j$ of $K$ in order from left to right (i.e., from largest to smallest); let $T$ be the resulting tableau. By construction we have $\wt(T) = \wt(K)$. Notice this is the inverse procedure to $\Psi_S$: for example, applying this to $S$ and $K=\Psi_S(T)$ in Example~\ref{ex:PsiSbijection} recovers $T$. 

This procedure replaces each run in $S$ with the corresponding row of $K$, weakly preserving the relative order between entries of $S$ in the same run. Moreover, the strict inequality between entries in different runs of $S$ is (strictly) preserved, since the smallest entry of a row of $K$ is strictly larger than the largest entry of any lower row of $K$, by (F2). From this, we immediately have that entries of $T$ weakly decrease along rows, so (SSRIF1) is satisfied. Now, since entries strictly increase up the first column of $S$, every entry in the first column of $S$ necessarily belongs to a different run. Since entries lower in the first column are used by runs with smaller indices,  
we have that the first column of $T$ is strictly increasing, so (SSRIF2) is satisfied. Entries in each row $p_j$ of $K$ are all weakly smaller than $p_j$ by (F2), and since $p_j$ is weakly smaller than the row index of the lowest entry of run $r_j$ of $S$, all entries of each run $r_j$ are replaced by entries that are weakly smaller than the row index of the lowest entry in $r_j$; hence (SSRIF3) is satisfied. Finally, since the smallest entry in row $p_{j+1}$ of $K$ is strictly larger than the largest entry of row $p_j$ of $K$, the replacement of entries of $S$ preserves the run structure, i.e., the runs of $S$ and runs of $T$ each involve the same collection of boxes taken in the same order. Hence $T$ standardizes to $S$. The uniqueness of $T$ with these properties follows from the lack of choice at each step. 
\qed

\begin{ex}\label{ex:PsiSbijection}
Consider $T \in \SSRIF(0,2,0,4,3)$ from Example~\ref{ex:SSRIF}, so that 
\[T \,\,\, {=} \,\,\,\vline\tableau{5 & 4 & 2 \\ 3 & 3 & 2 & 1 \\ \\  2 & 1 \\ \\\hline } \qquad \mbox{ standardizes to } \qquad S \,\,\, {=} \,\,\, \vline\tableau{  9 & 8 & 3 \\ 7 & 6 & 4 & 1 \\ \\ 5 & 2 \\ \\\hline }.\]  
Recall from Example~\ref{ex:wdes} that the run decomposition of $T$ is $54|33|222|11$ and that $\wdes_{\SSRIF}(T)=\wdes_{\SRIF}(S)=(2,3,0,2,2)$. Therefore, applying the bijection $\Psi_S$ from the proof of Theorem~\ref{thm:rdistoslide} gives
\[\Psi_S(T) \,\,\, {=} \,\,\, \vline\tableau{5 & 4 \\ 3 & 3 \\ \\ 2 & 2 & 2 \\ 1 & 1 \\\hline} \in \FSSF(\wdes_{\SRIF}(S)).\]
\end{ex}

\begin{ex}\label{ex:rdistofs}
We have 
\[\rdis_{(0,3,0,2)} = \fs_{(0,3,0,2)} + \fs_{(1,3,0,1)} + \fs_{(2,2,0,1)} + \fs_{(0,4,0,1)}\]
as computed by $\SRIF(0,3,0,2)$ (shown in Figure~\ref{fig:0302slide}), and their weak descent compositions, respectively, $(0,3,0,2)$, $(2,2,0,1)$, $(1,3,0,1)$, $(0,4,0,1)$.
\end{ex}

We note that the fundamental slide expansion of a reverse dual immaculate slide polynomial is not in general multiplicity-free, as illustrated by Example~\ref{ex:notmultfree} below.

\begin{ex}\label{ex:notmultfree}
Let $S,S'\in \SRIF(1,2,3)$, where
\[S \mbox{ $=$ } \vline\tableau{  6 & 5 & 1 \\ 4 & 3 \\  2 \\\hline } \qquad S' \mbox{ $=$ }   \vline\tableau{  6 & 5 & 3 \\ 4 & 1 \\  2 \\\hline }\]
Then $\wdes_{\SRIF}(S)=\wdes_{\SRIF}(S')=(2,2,2)$, but $S\neq S'$. Hence by Theorem~\ref{thm:rdistoslide} the coefficient of $\fs_{(2,2,2)}$ in the fundamental slide expansion of $\rdis_{(1,2,3)}$ is at least $2$.
\end{ex}

%%%%%%%%%%%%%%%%%%%%%%%%%%%%%%%%%%%%%%%%%%%%%%%%%%%%%%%%%%%%
\section{Reverse dual immaculate slide polynomials stabilize to reverse dual immaculate quasisymmetric functions}\label{sec:revstable}
%%%%%%%%%%%%%%%%%%%%%%%%%%%%%%%%%%%%%%%%%%%%%%%%%%%%%%%%%%%%

In this section, we first give a formula for the monomial expansion of reverse dual immaculate quasisymmetric functions. Then we establish conditions under which the reverse dual immaculate slide polynomials are quasisymmetric.  Next, we prove that the reverse dual immaculate slide polynomials limit to the reverse dual immaculate quasisymmetric functions. 
Finally, we show that the expansion of the reverse dual immaculate slide polynomials into fundamental slide polynomials (Theorem~\ref{thm:rdistoslide}) stabilizes to the expansion  (\ref{eq:direv}) of the reverse dual immaculate quasisymmetric functions into fundamental quasisymmetric functions.

Given a composition $\alpha$, recall $\SSRIT(\alpha)$ is all fillings of $D(\alpha)$ in which entries decrease from left to right along rows and strictly increase up the leftmost column (Definition~\ref{def:SSRIT}). Let $\SSRIT_m(\alpha)$ be the fillings in $\SSRIT(\alpha)$ whose largest entry does not exceed $m$. 

\begin{proposition}
Let $\alpha$ be a composition. Then
\[\rdi_\alpha = \sum_{T\in \SSRIT(\alpha)}x^{\wt(T)}.\]
\end{proposition}
\begin{proof}
For any $m>0$, we have a weight-reversing involution $\SSRIT_m(\alpha)\rightarrow \SSIT_m(\rev(\alpha))$ via replacing every entry $i$ of $T$ with $m+1-i$, and reversing the order of the rows. Since (by Theorem~\ref{thm:ditomonomial}) $\di_{\rev(\alpha)} (x_1,\hdots , x_m) = \sum_{T\in \SSIT_m(\rev(\alpha))}x^{\wt(T)}$, we have $\rdi_\alpha(x_1,\ldots , x_m) = \sum_{T\in \SSRIT_m(\alpha)}x^{\wt(T)}$. Since this holds for all $m>0$, the result follows by letting $m\to \infty$.
\end{proof}

Given a weak composition $a$, let $0^m \times a$ denote the weak composition formed by prepending $m$ zeros to $a$. For example, if $a=(1,0,2)$ then $0^3 \times a = (0,0,0,1,0,2)$. Let $\flatten(a)$ denote the composition obtained by removing all zero entries from $a$, e.g., $\flatten(0,0,0,1,0,2)=(1,2)$. We can now characterize exactly when a reverse dual immaculate slide polynomial is quasisymmetric.

\begin{proposition}\label{prop:rdisQS}
Let $a$ be a weak composition of length $\ell$. Then $\rdis_a$ is quasisymmetric in $x_1, \ldots, x_\ell$ if and only if $a$ has no zero entry to the right of a nonzero entry. In this case, 
\[\rdis_a = \rdi_{\flatten(a)}(x_1, \ldots , x_\ell).\]
\end{proposition}
\begin{proof}
Suppose $a$ has no zero entry to the right of a nonzero entry. It is enough to establish a weight-preserving bijection $\psi: \SSRIT_\ell(\flatten(a))\rightarrow \SSRIF(a)$. For $T\in \SSRIT_\ell(\flatten(a))$, let $\psi(T)$ be the the filling of shape $a$ obtained by letting the $i^{th}$ nonempty row of $\psi(T)$ (from the top) be the $i^{th}$ row of $T$ (from the top). Then $\psi(T)$ clearly satisfies (SSRIF1) and (SSRIF2) from Definition~\ref{def:SSRIF}. Since all entries in the $i^{th}$ row of $T$ from the top are at most $\ell+1-i$, and $a$ has no zero entry to the right of a nonzero entry, all entries in the $i^{th}$ row from the top of $\psi(T)$ are at most $\ell+1-i$, and so $\psi(T)$ satisfies (SSRIF3). This map is invertible, via removing all empty rows of elements of $\SSRIF(a)$. 

Conversely, suppose $a_i\neq 0$ but $a_{i+1} = 0$, for some $i<\ell$. Then $\rdis_a$ includes the monomial $x^a$, given by the semistandard reverse immaculate filling $T'$ whose entries in each row $i$ are all $i$. By the definition of a $\SSRIF$, no other $T\in \SSRIF(a)$ can have a larger entry than $T'$ does in any box. Therefore, since $a_i>a_{i+1}$, $\rdis_a$ does not have the monomial $x^{s_i(a)}$, where $s_i(a)$ is the weak composition obtained from $a$ by interchanging the entries $a_i$ and $a_{i+1}$. But since $a_{i+1}=0$, any quasisymmetric polynomial in variables $x_1,\hdots , x_m$ that contains $x^a$ must also contain $x^{s_i(a)}$.
\end{proof}

Note that if the zero entries in $a$ that appear after some nonzero entry all appear after the last nonzero entry of $a$, then $\rdis_a$ will be quasisymmetric in the truncated set of variables $x_1, \ldots , x_r$, where $a_r$ is the last nonzero part of $a$.

We are now ready to prove Theorem~\ref{thm:stablelimit} which we recall states that for a weak composition $a$,
\[\lim_{m\to \infty} \rdis_{0^m\times a} = \rdi_{\flatten(a)}.\]

\noindent\emph{Proof of Theorem~\ref{thm:stablelimit}:}  
We will show that $\rdis_{0^m\times a}(x_1, \ldots , x_m, 0, 0,  \ldots)$ is equal to $\rdi_{\flatten(a)}(x_1, \ldots , x_m, 0, 0,  \ldots)$, for any $m>0$. The theorem then follows by letting $m\to \infty$.

It suffices to establish a (weight-preserving) bijection \[\psi: \SSRIT_m(\flatten(a))\rightarrow \SSRIF_m(0^m\times a).\]
 
Suppose $a$ has nonzero entries in positions $n_1 < n_2 < \cdots < n_k$. Let $T\in \SSRIT_m(\flatten(a))$ and define $\psi(T)$ to be the filling of shape $0^m\times a$ obtained by letting the $n_i${th} row of $\psi(T)$ be the $i^{th}$ row of $T$. See Example~\ref{ex:psibijection} for an example. 
Clearly $\psi$ is injective and weight-preserving. 

We claim $\psi(T) \in \SSRIF_m(0^m\times a)$. Since $\psi$ only moves rows of $T$ while preserving their relative order, the only condition from Definition~\ref{def:SSRIF} that needs to be checked is that the entries in row $j$ of $\psi(T)$ do not exceed $j$. But this follows since all entries of $\psi(T)$ are at most $m$, and the lowest nonempty row of $0^m\times a$ has index $m+1$. 

For an inverse, consider the map $\phi: \SSRIF_m(0^m\times a)\rightarrow \SSRIT_m(\flatten(a))$ defined by removing all empty rows of the $\SSRIF S$. By definition, $\phi(S)$ is a tableau of shape $D(\flatten(a))$. Since the entries in the leftmost column of $S$ strictly increase and do not exceed $m$, the first column entry in row $n_i$ of $S$ (and thus every entry in row $n_i$ of $S$) does not exceed $i$. Hence $\phi(S)\in \SSRIT_m(\flatten(a))$. To see injectivity, suppose $S'\in \SSRIF_m(0^m\times a)$ such that $S'\neq S$.  In particular, suppose that in a certain box $S$ has entry $x$ and $S'$ has entry $y\neq x$. Since $\phi$ preserves entries, $\phi(S)$ will have entry $x$ and $\phi(S')$ entry $y$ in the same box of $D(\flatten(a))$. Hence $\phi$ is injective. It remains to note that $\phi$ is mutually inverse with $\psi$.
\qed

\begin{ex}\label{ex:psibijection}
If $a=(4,0,3)$ and $m=3$, the bijection $\psi$ from the proof of Theorem~\ref{thm:stablelimit} maps the $T\in \SSRIT_3(4,3)$ on the left below to $\psi(T)\in \SSRIF_3(0,0,0,4,0,3)$ on the right. 
\[
      \begin{array}{c@{\hskip2\cellsize}c@{\hskip2\cellsize}c@{\hskip2\cellsize}c}
        T=\tableau{ 3  & 3 & 2 \\ 2 & 2 & 1 & 1 }  & &
        \psi(T) \,\,\, {=} \,\,\, \vline\tableau{ 3  & 3 & 2 \\ \\ 2 & 2 & 1 & 1 \\ \\ \\ \\  \hline } \\
             \end{array}\]
\end{ex}

\begin{cor}\label{cor:stablelimit}
The formula of Theorem~\ref{thm:stablelimit} stabilizes to the formula (\ref{eq:direv})
\[\rdi_{\alpha} = \sum_{S\in \SRIT(\alpha)}F_{\des_{\SRIT}(S)}\]
for the expansion of the reverse dual immaculate quasisymmetric function $\rdi_{\flatten(a)}$ into fundamental quasisymmetric functions.
\end{cor}
\begin{proof}
Recall Theorem~\ref{thm:rdistoslide} states that for a weak composition $a$, we have 
\[\rdis_{a} = \sum_{S\in \SRIF(a)} \fs_{\wdes_{\SRIF}(S)}.\]
By Theorem~\ref{thm:stablelimit}, the left hand side of this stabilizes to $\rdi_{\flatten(a)}$. By \cite{Assaf.Searles:1} the fundamental slide polynomial $\fs_a$ stabilizes to the fundamental quasisymmetric function $F_{\flatten(a)}$, hence the right hand side stabilizes to $\sum_{S\in \SRIF(a)} F_{\flatten(\wdes_{\SRIF}(S))}$. 

For $S\in \SRIF(a)$, the nonzero parts of $\wdes_{\SRIF}(S)$ are by definition the lengths of the runs of $S$ (between descents), hence the flattening of $\wdes_{\SRIF}(S)$ gives the descent composition of $S$. Moreover there is a bijection between $\SRIF(a)$ and $\SRIT(\flatten(a))$, obtained by removing empty rows from elements of $\SRIF(a)$, which clearly preserves the descents. Hence we have 
\[\sum_{S\in \SRIF(a)} F_{\flatten(\wdes_{\SRIF}(S))} = \sum_{S\in \SRIT(\flatten(a))} F_{\des_{\SRIT}(S)}\]
as required.
\end{proof}

%%%%%%%%%%%%%%%%%%%%%%%%%%%%%%%%%%%%%%%%%%%%%%%%%%%%%%%%%%%%
\section{A standard formula for the fundamental slide expansion of a quasi-key polynomial}\label{sec:qktofs}
%%%%%%%%%%%%%%%%%%%%%%%%%%%%%%%%%%%%%%%%%%%%%%%%%%%%%%%%%%%%

In this section, we consider the quasi-key polynomials, which are a lifting of the quasisymmetric Schur basis for $\QSym$ to the polynomial ring \cite{Assaf.Searles:2}.
We give a new formula, in terms of \emph{standard} fillings, for the expansion of a quasi-key polynomial into the fundamental slide basis. This standard formula will be needed for the proofs in Section~\ref{sec:distoyqk}. 

Let $a$ be a weak composition. A \emph{triple} in $D(a)$ is a collection of three boxes in the positive quadrant with two adjacent in a row, and the third box in the same column as but in a higher row (possibly with boxes in between) than the box on the right. 
\begin{figure}[ht]
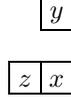

  \begin{displaymath}
    \begin{array}{l}
       \tableau{ & y } \\ \\ \tableau{  z & x }
    \end{array}
  \end{displaymath}
  \caption{Triples for weak composition diagrams.}\label{fig:reversetriples}
\end{figure}

To be a triple in $D(a)$, the boxes labeled $z$ and $y$ must be in $D(a)$, but the one labeled $x$ does not have to be. Given any filling of the boxes of $D(a)$ with integers, a triple in $D(a)$ is said to be an \emph{inversion triple} if whenever $y\le z$ we have $y<x$. If the box labeled $x$ in the triple is not in $D(a)$, it is assumed to have label $0$ for the purpose of deciding if the triple is inversion or not.

\begin{definition}{\label{def:RSSF}}
Given a weak composition $a$ of $n$, define a \emph{reverse semi-skyline filling} to be a filling of $D(a)$ such that
\begin{enumerate}[leftmargin=1.8cm]
\item[(RSSF1)] Row entries weakly decrease from left to right.
\item[(RSSF2)] Entries in the leftmost column strictly increase from bottom to top, and entries in any column are distinct.
\item[(RSSF3)] Entries in the $i^{th}$ row from the bottom do not exceed $i$. 
\item[(RSSF4)] All triples are inversion triples.
\end{enumerate}
A \emph{standard reverse skyline filling} of shape $a$ is a filling of $D(a)$ with $1,\ldots , n$ (each used once) that satisfies conditions (RSSF1), (RSSF2) and (RSSF4), but not necessarily (RSSF3). Let $\RSSF(a)$ (respectively, $\RSF(a)$) denote the set of reverse semi-skyline fillings (respectively, standard reverse skyline fillings) of shape $a$. 
\end{definition}

\begin{ex}
Figure~\ref{fig:QK0302} shows $\RSSF(a)$ for $a=(0,3,0,2)$.
\end{ex}

\begin{figure}[ht]
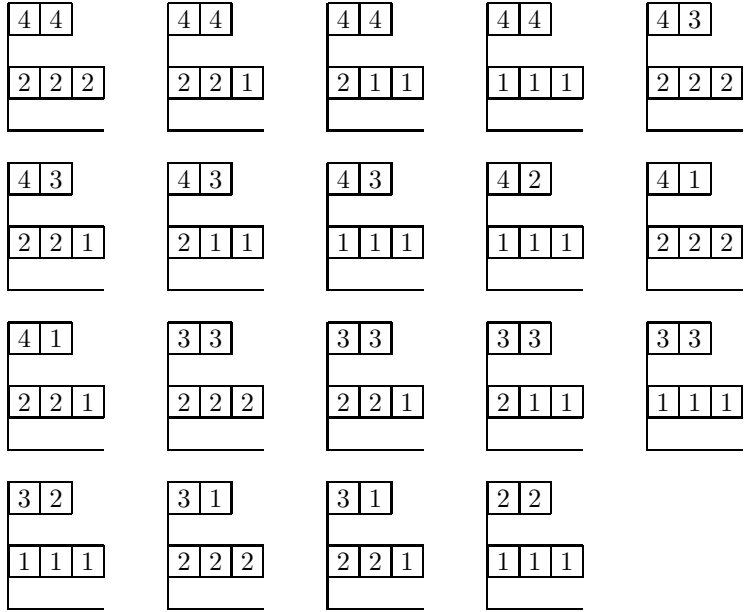

\begin{displaymath}
 \begin{array}{c@{\hskip2\cellsize}c@{\hskip2\cellsize}c@{\hskip2\cellsize}c@{\hskip2\cellsize}c@{\hskip2\cellsize}c@{\hskip2\cellsize}c}
 \vline \tableau{  4 & 4 \\ \\  2 & 2 & 2 \\ \\  \hline } &   \vline \tableau{  4 & 4 \\ \\  2 & 2 & 1 \\ \\  \hline }  &  \vline \tableau{  4 & 4 \\ \\  2 & 1 & 1 \\ \\  \hline }  &   \vline \tableau{  4 & 4 \\ \\  1 & 1 & 1 \\ \\  \hline }  &   \vline \tableau{  4 & 3 \\ \\  2 & 2 & 2 \\ \\  \hline }  \\ \\  
 \vline \tableau{  4 & 3 \\ \\  2 & 2 & 1 \\ \\  \hline } &   \vline \tableau{  4 & 3 \\ \\  2 & 1 & 1 \\ \\  \hline }  &  \vline \tableau{  4 & 3 \\ \\  1 & 1 & 1 \\ \\  \hline }  &   \vline \tableau{  4 & 2 \\ \\  1 & 1 & 1 \\ \\  \hline }  &   \vline \tableau{  4 & 1 \\ \\  2 & 2 & 2 \\ \\  \hline }  \\ \\   
  \vline \tableau{  4 & 1 \\ \\  2 & 2 & 1 \\ \\  \hline } &   \vline \tableau{  3 & 3 \\ \\  2 & 2 & 2 \\ \\  \hline }  &  \vline \tableau{  3 & 3 \\ \\  2 & 2 & 1 \\ \\  \hline }  &   \vline \tableau{  3 & 3 \\ \\  2 & 1 & 1 \\ \\  \hline }  &   \vline \tableau{  3 & 3 \\ \\  1 & 1 & 1 \\ \\  \hline }  \\ \\   
  \vline \tableau{  3 & 2 \\ \\  1 & 1 & 1 \\ \\  \hline }  &   \vline \tableau{  3 & 1 \\ \\  2 & 2 & 2 \\ \\  \hline }  &   \vline \tableau{  3 & 1 \\ \\  2 & 2 & 1 \\ \\  \hline }  &   \vline \tableau{  2 & 2 \\ \\  1 & 1 & 1 \\ \\  \hline }
    \end{array}
\end{displaymath}
\caption{The 19 reverse semi-skyline fillings of shape $(0,3,0,2)$.}\label{fig:QK0302}
\end{figure}

If condition (RSSF2) on reverse semi-skyline fillings is tightened to demand that entries in the leftmost column are equal to their row index, the resulting objects are the \emph{semi-skyline augmented fillings} from \cite{Mas08}, which generate the \emph{Demazure atom} basis for polynomials. The $\RSSF$s were first defined in \cite{Sea}, using a pair of triple conditions equivalent to (RSSF4), in order to give the following formula for the quasi-key polynomials $\qkey_a$: 

\begin{equation}\label{eq:qkey}
\qkey_a = \sum_{T\in \RSSF(a)}x^{\wt(T)}.
\end{equation}

\begin{ex}\label{ex:qkey0302}
For $a=(0,3,0,2)$, we have
\begin{align*}
\qkey_a & = x^{0302} + x^{1202}  +x^{2102} +x^{3002} +x^{0311} +x^{1211} +x^{2111} +x^{3011} +x^{3101} +x^{1301} \\
 & +x^{2201} +x^{0320} +x^{1220} +x^{2120} +x^{3020} +x^{3110} +x^{1310} +x^{2210} +x^{3200}, 
\end{align*}
which is computed from the reverse semi-skyline fillings shown in Figure~\ref{fig:QK0302}.
\end{ex}

\begin{remark}
Quasi-key polynomials were first defined in terms of \emph{quasi-Kohnert tableaux} \cite{Assaf.Searles:2}. An equivalent model called \emph{quasi-key tableaux}, which are fillings of $D(a)$, was introduced in \cite{Sea}. Although both quasi-key tableaux and RSSFs are fillings of $D(a)$ that generate the quasi-key polynomials, these families of fillings are different. Section 5 of \cite{Sea} is devoted to establishing a bijection between them.
\end{remark}

Analogously to Theorem~\ref{thm:stablelimit}, the stable limit of the quasi-key polynomial associated to the weak composition $a$ is the quasisymmetric Schur function associated to $\flatten(a)$:

\begin{theorem}\cite{Assaf.Searles:2}\label{thm:reversestablelimit}
Let $a$ be a weak composition of length $\ell$. Then 
\[\lim_{m\to \infty} \qkey_{0^m\times a} = \qs_{\flatten(a)}.\]
\end{theorem}

The reading word of a reverse semi-skyline filling or standard reverse skyline filling is the word obtained by reading the entries up columns, starting at the rightmost column and proceeding to the leftmost column. This order on entries is called the \emph{reading order}. 
The descent set of a standard reverse skyline filling $S$, denoted by $\des_{\RSF}(S)$, is the set of all $i$ such that $i+1$ is weakly to the right of $i$. 

\begin{remark}
It is possible that a filling $S$ of shape $a$ may belong to both $\SRIF(a)$ and $\RSF(a)$. However, we emphasise that $\des_{\SRIF}(S)$ and $\des_{\RSF}(S)$ are typically not equal.
\end{remark}

\begin{definition}\label{def:wdesqkey}
Let $a$ be a weak composition of $n$ of length $\ell$. The \emph{weak descent composition} of $S\in \RSF(a)$, denoted $\wdes_{\RSF}(S)$, is the weak composition of length $\ell$ obtained as follows. Decompose the word $n\, n-1 \, \ldots 2\, 1$ into runs by placing a bar between $i+1$ and $i$ whenever $i\in \des_{\RSF}(S)$. Suppose there are $k$ runs, i.e., the run decomposition yields $r_k | r_{k-1} | \ldots | r_1$.

Define a strictly decreasing sequence of numbers $p_k, \ldots , p_1$ recursively as follows. Let $p_k$ be the row index of the box containing $n$.  
Now, let $j<k$ and suppose $p_k, \ldots , p_{j+1}$ are known. Define $p_j$ to be the smaller of $p_{j+1}-1$ and the index of the lowest row containing an entry from the $j^{th}$ run $r_j$.

If any $p_j$ is nonpositive, then $\wdes_{\RSF}(S)=\emptyset$. Otherwise, $\wdes_{\RSF}(S)$ is the weak composition of length $\ell$ whose $p_j^{th}$ part is the number of entries in the run $r_j$, and whose other parts are all zero. 
\end{definition}

\begin{remark}
Notice that the definition of $\wdes_{\RSF}(S)$ is identical to the definition of $\wdes_{\SRIF}(S)$ (Definition~\ref{def:wdesrdis}), except that the run decomposition of $n \, n-1 \ldots 2 \, 1$ is determined by when $i\in \des_{\RSF}(S)$, as opposed to $\des_{\SRIF}(S)$.
\end{remark}

\begin{remark}\label{rmk:wdessemistandardqkey1}
Definition~\ref{def:wdesqkey} extends straightforwardly to give runs $r_j$, numbers $p_j$, and thus a weak descent composition $\wdes_{\RSSF}(T)$ for a reverse semi-skyline filling $T$. 
This is exactly analogous to the definition of $\wdes_{\SSRIF}(T)$ in Remark~\ref{rmk:wdessemistandard}.
\end{remark}

\begin{ex}\label{ex:SF}
The standard reverse skyline fillings of shape $(0,3,0,2)$ are shown in Figure~\ref{fig:0302SF} below.
\end{ex}

\begin{figure}[ht]
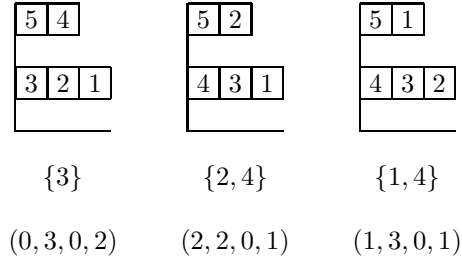

  \begin{center}
    \begin{displaymath}
      \begin{array}{c@{\hskip2\cellsize}c@{\hskip2\cellsize}c@{\hskip2\cellsize}c}
        \vline\tableau{  5 & 4  \\ \\  3 & 2 & 1 \\ \\\hline } &
        \vline\tableau{  5 & 2  \\ \\  4 & 3 & 1 \\ \\\hline } &
        \vline\tableau{  5 & 1  \\ \\  4 & 3 & 2 \\ \\\hline } \\ \\
        \{3\} & \{2,4\} & \{1,4\} \\ \\
        (0,3,0,2) & (2,2,0,1) & (1,3,0,1)
                     \end{array}
    \end{displaymath}
  \end{center}
  \caption{The 3 $\RSF$s for $a=(0,3,0,2)$, their descent sets (middle), and their weak descent compositions (bottom).}\label{fig:0302SF}
  \end{figure}

The \emph{standardization} of a reverse semi-skyline filling $T\in \RSSF(a)$, denoted $\std_{\RSSF}(T)\in \RSF(a)$, is obtained by replacing the $i^{th}$ smallest entry $i$ of $T$ by $i$ for all $1\le i \le n$, where if two entries are equal the entry appearing earlier in reading order is considered to be smaller. Since standardization preserves the relative order between entries, it is immediate that (RSSF1), (RSSF2) and (RSSF4) are preserved, hence $\std_{\RSSF}(T)\in \RSF(a)$.

The following lemma is proved in an identical manner to Lemma~\ref{lem:wdespreservedstd}

\begin{lemma}\label{lem:wdespreservedqkey}
Standardization preserves weak descent compositions; that is, if $\std_{\RSSF}(T)=S$ then $\wdes_{\RSSF}(T) = \wdes_{\RSF}(S)$.
\end{lemma}

We are now ready to establish that the fundamental slide expansion of a quasi-key polynomial is indexed by the standard reverse skyline fillings. 

\begin{lemma}\label{lem:stdeqqkey}
Let $a$ be a weak composition of length $\ell$. Then standardization partitions $\RSSF(a)$ into equivalence classes, in particular,
\[\RSSF(a) = \bigsqcup_{S\in \RSF(a)} \left\{T\in \RSSF(a) : \std_{\RSSF}(T) = S\right\}.\]
\end{lemma}
\begin{proof}
By the definition of standardization, given $T\in \RSSF(a)$ there is a unique $S\in \RSF(a)$ such that $\std_{\RSSF}(T)=S$.
\end{proof}

\begin{lemma}\label{lem:bothsidesemptyqkey}
Let $a$ be a weak composition and $S\in \RSF(a)$. Then $\wdes_{\RSF}(S)=\emptyset$ if and only if there is no $T\in \RSSF(a)$ that standardizes to $S$.
\end{lemma}
\begin{proof}
The proof that existence of a $T\in \RSSF(a)$ that standardizes to $S$ implies $\wdes_{\RSF}(S)\neq \emptyset$ is essentially identical to the argument in the first paragraph of the proof of Lemma~\ref{lem:bothsidesempty}. The converse direction is essentially the same as the argument in the second and third paragraphs of the proof of Lemma~\ref{lem:bothsidesempty}, using the same construction of $T$ as the filling obtained by replacing every entry of run $r_j$ of $S$ with the value $p_j$. However a few more details need to be shown to justify that the $T$ obtained is an $\RSSF$, specifically for (RSSF2) that columns of $T$ have no repeated entries, and for (RSSF4) that all triples in $T$ are inversion triples. We provide these details here. 

Any two entries in the same column of $S$ necessarily belong to different runs; this is immediate from the definition of descent set for $\RSF$'s. This means no two entries of any given column of $T$ correspond to same $p_j$ value. Since the $p_j$'s are distinct, all entries of any column of $T$ are distinct.

For (RSSF4), consider an inversion triple in $S$ (Figure~\ref{fig:reversetriples}).   If $y>z$, then $y$ and $z$ are in different runs since $z$ is to the left of $y$. Hence the entry replacing $z$ is strictly smaller than the entry replacing $y$, so this triple is still an inversion triple in $T$. If on the other hand we have $y<z$, then by the definition of an inversion triple we also have $y<x$. Since $y$ and $x$ are in the same column they must belong to different runs, with $y$ in a smaller-indexed run than the run containing $x$. Therefore, the entry replacing $y$ is strictly smaller than the entry replacing $x$, so this triple is still an inversion triple in $T$. Hence (RSSF4) is satisfied. 
\end{proof}

\begin{theorem}\label{thm:qkeytoslide}
Let $a$ be a weak composition of length $\ell$. Then
\[\qkey_{a} = \sum_{S\in \RSF(a)} \fs_{\wdes_{\RSF}(S)}.\]
\end{theorem}
\begin{proof}
For $S\in \RSF(a)$, define a map 
\[\Psi_S: \left\{T\in \RSSF(a) : \std_{\RSSF}(T) = S\right\} \,\,\,\, \rightarrow \,\,\,\, \FSSF(\wdes_{\RSF}(S))\]
in the same way as in the proof of Theorem~\ref{thm:rdistoslide}. The proof then proceeds in an essentially identical manner to that of Theorem~\ref{thm:rdistoslide}, using Lemmas~\ref{lem:wdespreservedqkey} and \ref{lem:bothsidesemptyqkey}, however a few more details need to be shown to establish that the $T$ constructed from a given $S\in \RSF(a)$ and $K\in \FSSF(\wdes_{\RSF}(S))$ is in fact an $\RSSF$. Specifically, we need to show for (RSSF2) that columns of $T$ have no repeated entries, and for (RSSF4) that all triples in $T$ are inversion triples. We provide these details here.

Any two entries in the same column of $S$ necessarily belong to different runs; this is immediate from the definition of descent set for $\RSF$s. Entries in runs of $S$ are replaced by entries in rows of $K$, and entries in a given row of $K$ are strictly larger than all entries in any lower row of $K$, by (F2). Hence replacing runs of $S$ by rows of $K$ ensures that strict inequality between entries in the same column is preserved. Therefore, since $S$ (being standard) has no repeated entry in any column, $T$ also has no repeated entry in any column.

For (RSSF4), we need to ensure that replacing runs of $S$ with rows of $K$ doesn't cause any inversion triple in $S$ to become a non-inversion triple in $T$. There are two possibilities to consider. Given a triple of entries of $S$ with $y>z$ (see Figure~\ref{fig:reversetriples}), suppose $y'$ replaces $y$ and $z'$ replaces $z$ when constructing $T$. We claim that $y'>z'$, and thus the triple is an inversion triple in $T$.  To see this, observe that since $y>z$ and $y$ is to the right of $z$, $y$ and $z$ belong to different runs in $S$. Hence strict inequality between these two entries are preserved when the runs are replaced by rows of $K$.  
Now suppose we have an inversion triple in $S$ with $y<z$, i.e. also $y<x$. It is enough to show that $y'<x'$ in $T$. But this follows since $y$ and $x$ are in the same column, hence in different runs.  Again, this replacement preserves the strict inequality.  
\end{proof}

%%%%%%%%%%%%%%%%%%%%%%%%%%%%%%%%%%%%%%%%%%%%%%%%%%%%%%%%
\section{Dual immaculate slide polynomials expand positively into Young quasi-key polynomials}\label{sec:distoyqk}
%%%%%%%%%%%%%%%%%%%%%%%%%%%%%%%%%%%%%%%%%%%%%%%%%%%%%%%%

In this section, we work with the Young versions of reverse dual immaculate slide polynomials and quasi-key polynomials, in order to apply and extend the combinatorics developed by \cite{AHM18}. We give positive formulas for the decompositions of both (Young) dual immaculate slide polynomials and Young quasi-key polynomials into sums of Young fundamental slide polynomials.  From the former, we recover the formula of \cite{BBSSZ14} for the expansion of a dual immaculate quasisymmetric function into fundamental quasisymmetric functions. We then use our decomposition formulas, together with an insertion algorithm, to prove a positive formula for the expansion of dual immaculate slide polynomials into Young quasi-key polynomials. As a corollary, we obtain a similar formula for the expansion of reverse dual immaculate slide polynomials into quasi-key polynomials (which limits to an expansion of the reverse dual immaculate quasisymmetric functions into quasisymmetric Schur functions).  We also recover the formula of \cite{AHM18} for the expansion of a dual immaculate quasisymmetric function into Young quasisymmetric Schur functions (Theorem~\ref{thm:ditoyqs}) and establish that reverse dual immaculate slide polynomials expand positively in the Demazure atom basis of \cite{LasSch90} and~\cite{Mas09}.

\subsection{Young versions}
Given a weak composition $a$ of $n$ into $\ell$ parts, define the \emph{(Young) dual immaculate slide polynomial} $\dis_a$ to be the reverse dual immaculate slide polynomial for $\rev(a)$ with the variable set reversed, i.e.,
\begin{equation}\label{eq:dis}
\dis_a(x_1,\ldots , x_\ell) = \rdis_{\rev(a)}(x_\ell, \ldots , x_1). 
\end{equation}
For example,
\[\dis_{(2,0,1)}(x_1,x_2,x_3) = \rdis_{(1,0,2)}(x_3,x_2, x_1) = x^{201}+x^{111}+x^{102}+x^{021}+x^{012}.\] 

\begin{remark}
With Young families of polynomials, one must be careful regarding the length of the weak composition, in a way that isn't necessary for reverse families. Specifically, assuming that the number of variables equals the length of the weak composition, appending zeros to a weak composition does not change a reverse polynomial, but it does change a Young polynomial. For example, $\rdis_a = \rdis_{a\times 0^m}$ for any $m$, but $\dis_{a\times 0^m}$ is a different polynomial for every value of $m$.
\end{remark}

\begin{definition}{\label{def:SSIF}}
A \emph{semistandard immaculate filling} of weak composition shape $a$ is a filling of the boxes in the diagram of $a$ with positive integers satisfying the following properties.
\begin{enumerate}[leftmargin=1.8cm] 
\item[(SSIF1)] Row entries weakly increase from left to right.
\item[(SSIF2)] Entries in the leftmost column strictly increase from bottom to top.
\item[(SSIF3)] Entries in the $i^{th}$ row from the bottom are no smaller than $i$.
\end{enumerate}
A \emph{standard immaculate filling} of shape $a$ is a filling of $D(a)$ with $1, \ldots , n$ (each used once) that satisfies conditions (SSIF1) and (SSIF2) but not necessarily (SSIF3). Let $\SSIF(a)$ (respectively, $\SIF(a)$) denote the set of all semistandard (respectively, standard) immaculate fillings of shape $a$.
\end{definition}

The reading word of a semistandard or standard immaculate filling is given by reading the entries along rows from left to right, starting at the top row and proceeding downwards. This order on entries is called the \emph{reading order}. For $a$ a weak composition of $n$ of length $\ell$, notice the reading words of $\SSIF(a)$ are exactly the reading words of $\SSRIF(\rev(a))$, reversed and with entries $i$ replaced by $\ell+1-i$, and the reading words of $\SIF(a)$ are exactly the reading words of $\SRIF(\rev(a))$, reversed and with entries $i$ replaced by $n+1-i$. The \emph{descent set} of $S\in \SIF(a)$, denoted $\des_{\SIF}(a)$, is set of all $i$ such that $i+1$ is in a strictly higher row.  (Note that this is the same as the descent set for a $\SRIF$.)  

We also define weak descent compositions for standard immaculate fillings. 
\begin{definition}\label{def:wdesdis}
Let $a$ be a weak composition of $n$ of length $\ell$. The \emph{weak descent composition} of $S\in \SIF(a)$, denoted $\wdes_{\SIF}(S)$, is the weak composition of length $\ell$ obtained as follows. First, decompose the word $1 \, 2 \ldots n-1 \, n$ into runs by placing a bar between $i$ and $i+1$ whenever $i\in \des_{\SIF}(S)$. Suppose there are $k$ runs, i.e., the run decomposition yields $r_1|r_2|\ldots |r_k$. 

Define a strictly increasing sequence of numbers $p_1, p_2, \ldots , p_k$ recursively as follows. Let $p_1$ be the row index of the box containing the entry $1$. Now let $j>1$ and suppose $p_1, \ldots , p_{j-1}$ are known. Define $p_j$ to be the larger of $p_{j-1}+1$ and the index of the highest row containing an entry of the $j^{th}$ run $r_j$. If any $p_j$ is strictly greater than $\ell$, then $\wdes_{\SIF}(S)=\emptyset$. Otherwise $\wdes_{\SIF}(S)$ is the weak composition of length $\ell$ whose $p_j^{th}$ part is the number of entries in the run $r_j$, and whose other parts are all zero.
\end{definition}

Notice the symmetry between Definition~\ref{def:wdesdis} of weak descent composition for $S\in \SIF(a)$ and Definition~\ref{def:wdesrdis} of weak descent composition for $S\in \SRIF(a)$.

\begin{ex}
Let $a=(2,0,3,0)$. The $\SIF$s for $a=(2,0,3,0)$ and their weak descent compositions are in Figure~\ref{fig:2030SIF1}. Compare to Figure~\ref{fig:0302slide} which shows the $\SRIF$s for $\rev(a)=(0,3,0,2)$.
\end{ex}

\begin{figure}[ht]
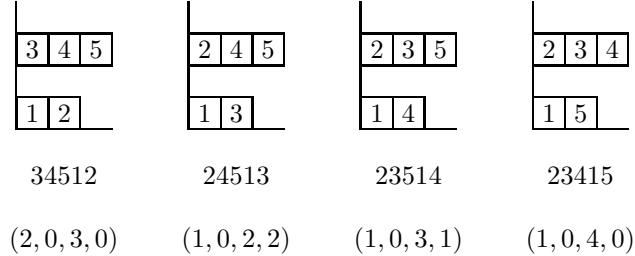

  \begin{center}
    \begin{displaymath}
      \begin{array}{c@{\hskip2\cellsize}c@{\hskip2\cellsize}c@{\hskip2\cellsize}c}
        \vline\tableau{ \\  3 & 4 & 5 \\ \\  1 & 2  \\\hline } &
        \vline\tableau{ \\  2 & 4 & 5 \\ \\  1 & 3  \\\hline } &
        \vline\tableau{ \\  2 & 3 & 5 \\ \\  1 & 4  \\\hline } &
        \vline\tableau{ \\  2 & 3 & 4 \\ \\  1 & 5  \\\hline } \\ \\
         34512 & 24513 & 23514 & 23415 \\ \\
        (2,0,3,0) & (1,0,2,2) & (1,0,3,1) & (1,0,4,0)
                     \end{array}
    \end{displaymath}
  \end{center}
  \caption{The 4 $\SIF$s for $a=(2,0,3,0)$, their reading words, and their weak descent compositions.}\label{fig:2030SIF1}
  \end{figure}

Given a weak composition $a$ of length $\ell$, we define a ``Young'' version, $\yfs_a$, of the fundamental slide polynomial $\fs_a$ by setting 
\begin{equation}\label{eq:yfs}
\yfs_a(x_1,\ldots , x_\ell) = \fs_{\rev(a)}(x_\ell, \ldots , x_1).
\end{equation}

We now describe the decomposition of the dual immaculate slide polynomials into Young fundamental slide polynomials.

\begin{theorem}\label{thm:ydistoyslide}
Let $a$ be a weak composition of $n$, of length $\ell$. Then
\[\dis_{a} = \sum_{S\in \SIF(a)} \yfs_{\wdes_{\SIF}(S)}.\]
\end{theorem}
\begin{proof}
The map $\theta:\SIF(a)\rightarrow \SRIF(\rev(a))$ which reverses the order of the rows and replaces each entry $i$ with $n+1-i$ is an involution; cf. Lemma~\ref{lem:SITSRITbij}. Moreover, it is straightforward that a consecutive sequence $i, i+1, \ldots j$ is a run in $S\in\SIF(a)$ if and only if $n+1-j, \ldots , n-i, n-i+1$ is a run in $\theta(S)$.  Similarly, if the highest entry in this run in $S$ is in row $r$, then the lowest entry in the corresponding run in $\theta(S)$ is in row $\ell+1-r$. It then follows from the symmetry between the definitions of $\wdes_{\SIF}$ and $\wdes_{\SRIF}$ that this involution is $\wdes$-reversing. Therefore,
\begin{align*}
\dis_a(x_1,\ldots , x_\ell) & = \rdis_{\rev(a)}(x_\ell, \ldots , x_1) \\
                                       & = \sum_{S\in \SRIF(\rev(a))}\fs_{\rev(\wdes_{\SRIF}(S))}(x_\ell, \ldots , x_1) \\
                                       & = \sum_{S\in \SIF(a)}\yfs_{\wdes_{\SIF}(S)}(x_1,\ldots , x_\ell).
\end{align*}   
\end{proof}

A similar argument to the proof of Proposition~\ref{prop:rdisQS} establishes the following:

\begin{proposition}\label{prop:disQS}
Let $a$ be a weak composition of length $\ell$. Then $\dis_a$ is quasisymmetric in $x_1, \ldots, x_\ell$ if and only if $a$ has no zero entry to the left of a nonzero entry. In this case, 
\[\dis_a = \di_{\flatten(a)}(x_1, \ldots , x_\ell).\]
\end{proposition}

\begin{remark}\label{rmk:noyoungstability}
As alluded to in the introduction, there are certain notions that work well with either the Young version or the reverse 
version of a family of polynomials, but not both. Illustrating this, there is no analogue of Theorem~\ref{thm:stablelimit} for the dual immaculate slide polynomials. One can define a stable limit of $\dis_a$ by \emph{appending} $m$ zeros to $a$ and letting $m\to \infty$, however this does not recover $\di_{\flatten(a)}$ unless $a$ has no zero entry to the left of a nonzero entry. In that case $\dis_a$ was already equal to $\di_{\flatten(a)}(x_1, \ldots , x_\ell)$ by Proposition~\ref{prop:disQS}. 
\end{remark}

Despite Remark~\ref{rmk:noyoungstability}, we may still recover the Young analogue of Corollary~\ref{cor:stablelimit} from Theorem~\ref{thm:ydistoyslide}.  Corollary~\ref{cor:dislift} below (the Young analogue of Corollary~\ref{thm:stablelimit}) is exactly Proposition~\ref{prop:ditof}, proved in \cite{BBSSZ14}, on the expansion of dual immaculate quasisymmetric functions into fundamental quasisymmetric functions. We include it as a corollary since it establishes that the finite-variable version of Proposition~\ref{prop:ditof} is a special case of Theorem~\ref{thm:ydistoyslide}, and therefore Theorem~\ref{thm:ydistoyslide} extends this expansion to the polynomial ring.

\begin{cor}\label{cor:dislift}
Let $\alpha$ be a composition. Then for any positive integer $m$ greater than or equal to the number of parts of $\alpha$, we obtain the formula \[\di_{\alpha}(x_1, \ldots , x_m) = \sum_{S\in \SIT(\alpha)}F_{\des_{\SIT}(S)}(x_1,\ldots , x_m)\]
as a special case of Theorem~\ref{thm:ydistoyslide}. 
Then, letting $m \rightarrow \infty$ recovers Proposition~\ref{prop:ditof}.
\end{cor}
\begin{proof}
From Theorem~\ref{thm:ydistoyslide}, we have
\begin{equation}\label{eqn:distoyfs}
\dis_{\alpha \times 0^{m-\ell(\alpha)}} = \sum_{S\in \SIF(\alpha \times 0^{m-\ell(\alpha)}) }\yfs_{\wdes_{\SIF}(S)}.
\end{equation}
Note that $\alpha \times 0^{m-\ell(\alpha)}$ is the weak composition of length $m$ obtained by appending the appropriate number of zeros to $\alpha$.

By Proposition~\ref{prop:disQS}, the left hand side of (\ref{eqn:distoyfs}) is equal to $\di_{\alpha}(x_1, \ldots , x_m)$. For the right hand side, first note that $\SIF(\alpha \times 0^{m-\ell(\alpha)})$ is clearly in (descent set-preserving) bijection with $\SIT(\alpha)$ for any $m$, simply by ignoring the $m-\ell(\alpha)$ empty rows above the shape of $\alpha$. 

Next, we show that for $S\in \SIF(\alpha \times 0^{m-\ell(\alpha)})$, the weak descent composition $\wdes_{\SIF}(S)$ consists of the descent composition of $S$, with zeros appended to make it a weak composition of length $m$. To see this, suppose $S$ has $k$ runs, i.e., the descent composition of $S$ has length $k$. Recall the nonzero parts of $\wdes_{\SIF}(S)$ are by definition exactly the lengths of the increasing runs in $S$ (between descents). Now note that by the increasing first column condition, the first entry of every row of $S$ is at least the index of that row, and then by the increasing row condition we have that every entry in $S$ is at least the index of the row it occupies. It follows that $p_1=1$ and $p_j = p_{j-1}+1$ for every $1<j\le k$, i.e., the first $k$ entries of $\wdes_{\SIF}(S)$ are exactly the descent composition of $S$, and the remaining entries are zero.

Finally, for any composition $\beta$, we have 
\begin{align*}
\yfs_{\beta\times 0^{m-\ell(\beta)}}(x_1, \ldots, x_m) & = \fs_{0^{m-\ell(\beta)}\times \rev(\beta)} (x_m, \ldots, x_1) \\
									     & = F_{\rev(\beta)}(x_m, \ldots , x_1) \\ 
									     & = F_{\beta}(x_1, \ldots , x_m)
\end{align*}
where the first equality is by definition of $\yfs$, the second is proved in \cite{Assaf.Searles:1}, and the third is immediate from the definition of fundamental quasisymmetric polynomials. In particular, letting $\beta$ be the descent composition of $S$, we have that the right hand side of (\ref{eqn:distoyfs}) is equal to $\sum_{S\in \SIT(\alpha)}F_{\des_{\SIT}(S)}(x_1,\ldots , x_m)$, as required.
\end{proof}

One may analogously define a ``Young'' version of the quasi-key polynomial. Let $a$ be a weak composition of length $\ell$. The \emph{Young quasi-key polynomial} $\yqk_a$ is defined by
\begin{align}{\label{eq:yqktoqk}}
\yqk_a(x_1, \ldots , x_\ell) &= \qkey_{\rev(a)}(x_\ell , \ldots , x_1).
\end{align}
We note that if $a$ is in fact a composition, then (\ref{eq:yqktoqk}) reduces to the relationship between Young quasisymmetric Schur polynomials and quasisymmetric Schur polynomials.

We define an increasing skyline fillings model for Young quasi-key polynomials, which we will use for our weak insertion algorithm. This requires the concept of a Young triple. Let $a$ be a weak composition. A \emph{Young triple} in $D(a)$ is a collection of three boxes in the positive quadrant with two adjacent in a row, and the third box in the same column as the box on the right but in a lower row, possibly with boxes or empty spaces between. 
\begin{figure}[ht]
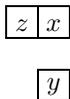

  \begin{displaymath}
    \begin{array}{l}
       \tableau{ z & x } \\ \\ \tableau{   & y }
    \end{array}
  \end{displaymath}
  \caption{Young triples for weak composition diagrams.}\label{fig:triples}
\end{figure}

To be a Young triple in $D(a)$, the boxes labeled $z$ and $y$ must be in $D(a)$, but the box labeled $x$ does not have to be. If the box labeled $x$ in the Young triple does not exist in $D(a)$, it is assumed to have label $\infty$ for the purpose of deciding whether the triple is an inversion triple.  Given any filling of the boxes of $D(a)$ with integers, a Young triple in $D(a)$ is said to be an \emph{inversion triple} if whenever $y\ge z$, we have $y>x$. 

\begin{definition}{\label{def:YSSF}}
A \emph{Young semi-skyline filling} of shape $a$ is a filling of the boxes of $D(a)$ with positive integers such that 
\begin{enumerate}[leftmargin=1.8cm]
\item[(YSSF1)] Row entries weakly increase from left to right
\item[(YSSF2)] Entries in the leftmost column strictly increase from bottom to top, and entries in any column are distinct
\item[(YSSF3)] No entry in row $i$ is smaller than $i$
\item[(YSSF4)] All Young triples are inversion triples.
\end{enumerate}
A \emph{Young standard skyline filling} of shape $a$ is a filling of $D(a)$ with $1, \ldots , n$ (each used once) that satisfies (YSSF1), (YSSF2) and (YSSF4), but not necessarily (YSSF3). Let $\YSSF(a)$ (respectively, $\YSF(a)$) denote the set of all Young semi-skyline fillings (respectively, Young standard skyline fillings) of shape $a$. 
\end{definition}

\begin{proposition}
Let $a$ be a weak composition of length $\ell$. Then
\[\yqk_a = \sum_{T\in \YSSF(a)}x^{\wt(T)}.\]
\end{proposition}
\begin{proof}
The map between $\YSSF(a)$ and $\RSSF(\rev(a))$ which reverses the order of the rows and replaces each entry $i$ with $\ell+1-i$, where $\ell$ is the length of $a$, is a weight-reversing involution. The formula then follows from the definition of quasi-key polynomials in terms of $\RSSF$s and the definition (\ref{eq:yqktoqk}) of Young quasi-key polynomials in terms of quasi-key polynomials.
\end{proof}

The reading word of a Young semi-skyline filling is given by reading the entries down columns, starting at the rightmost column and proceeding leftward. This order on entries is called the \emph{reading order}. For $a$ a weak composition of $n$ of length $\ell$, notice the reading words of $\YSSF(a)$ are exactly the reading words of $\RSSF(\rev(a))$ with each entry $i$ replaced by $\ell+1-i$, and the reading words of $\YSF(a)$ are exactly the reading words of $\RSF(\rev(a))$ reversed and with entries $i$ replaced by $n+1-i$. 
The \emph{descent set} of $S\in \YSF(a)$, denoted $\des_{\YSF}(a)$, is set of all $i$ such that $i+1$ is weakly to the left of $i$.  (This is not the same as the descent set of an $\RSF$, where $i\in \des_{\RSF}(a)$ if and only if $i+1$ is weakly to the \emph{right} of $i$.)

We now define the weak descent composition associated to a Young standard skyline filling.

\begin{definition}\label{def:wdesyqk}
Let $a$ be a weak composition of $n$ of length $\ell$. The \emph{weak descent composition} of $S\in \YSF(a)$, denoted $\wdes_{\YSF}(S)$, is the weak composition of length $\ell$ obtained as follows. First, decompose the word $1 \, 2 \ldots n-1 \, n$ into runs by placing a bar between $i$ and $i+1$ whenever $i\in \des_{\YSF}(S)$. Suppose there are $k$ runs, i.e., the run decomposition yields $r_1|r_2|\ldots |r_k$. 

Define a strictly increasing sequence of numbers $p_1, p_2, \ldots , p_k$ recursively as follows. Let $p_1$ be the row index of the box containing the entry $1$. Now let $j>1$ and suppose $p_1, \ldots , p_{j-1}$ are known. Define $p_j$ to be the larger of $p_{j-1}+1$ and the index of the highest row containing an entry of the $j^{th}$ run $r_j$. If any $p_j$ is strictly greater than $\ell$, then $\wdes_{\YSF}(S)=\emptyset$. Otherwise $\wdes_{\YSF}(S)$ is the weak composition of length $\ell$ whose $p_j^{th}$ part is the number of entries in the run $r_j$, and whose other parts are all zero.
\end{definition}

Notice this definition is identical to Definition~\ref{def:wdesdis} of weak descent composition for standard immaculate fillings, except we use $\des_{\YSF}(S)$ rather than $\des_{\SIF}(S)$ to determine the run decomposition.

\begin{ex}\label{ex:YSF}
The Young standard skyline fillings of shape $a=(2,0,3,0)$ and their weak descent compositions are shown in Figure~\ref{fig:2030YSF} below. Compare to Example~\ref{ex:SF}, which shows the $\RSF$s for $\rev(a)=(0,3,0,2)$.
\end{ex}

  \begin{figure}[ht]
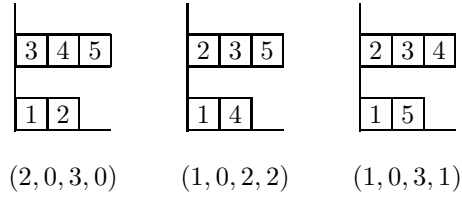

  \begin{center}
    \begin{displaymath}
      \begin{array}{c@{\hskip2\cellsize}c@{\hskip2\cellsize}c@{\hskip2\cellsize}c}
        \vline\tableau{ \\  3 & 4 & 5 \\ \\  1 & 2  \\\hline } &
        \vline\tableau{ \\  2 & 3 & 5 \\ \\  1 & 4  \\\hline } &
        \vline\tableau{ \\  2 & 3 & 4 \\ \\  1 & 5  \\\hline } \\ \\
        (2,0,3,0) & (1,0,2,2) & (1,0,3,1) 
                     \end{array}
    \end{displaymath}
  \end{center}
  \caption{The three $\YSF$s of shape $(2,0,3,0)$ and their weak descent compositions.}\label{fig:2030YSF}
  \end{figure}

\begin{theorem}\label{thm:yqktoyslide}
Let $a$ be a weak composition of length $\ell$. Then
\[\yqk_{a} = \sum_{S\in \YSF(a)} \yfs_{\wdes_{\YSF}(S)}.\]
\end{theorem}

\begin{proof}
This is proved similarly to Theorem~\ref{thm:ydistoyslide}. The same map $\theta$ which reverses the order of the rows and replaces each entry $i$ with $n+1-i$ is an involution between $\YSF(a)$ and $\RSF(\rev(a))$. This involution is similarly $\wdes$-reversing; notice that in this case we have $i+1$ weakly left of $i$ in $S\in \YSF(a)$ if and only if $n-i+1$ is weakly \emph{right} of $n-i$ in $\theta(S)$.  This is precisely what we need based on the relationship between descents in $\YSF$s and descents in $\RSF$s.
Therefore,
\begin{align*}
\yqk_a(x_1,\ldots , x_\ell) & = \qkey_{\rev(a)}(x_\ell, \ldots , x_1) \\
                                       & = \sum_{S\in \RSF(\rev(a))}\fs_{\rev(\wdes_{\RSF}(S))}(x_\ell, \ldots , x_1) \\
                                       & = \sum_{S\in \YSF(a)}\yfs_{\wdes_{\YSF}(S)}(x_1,\ldots , x_\ell).
\end{align*}  
 \end{proof}

\begin{ex}
From Figure~\ref{fig:2030YSF}, we compute
\[\yqk_{(2,0,3,0)} = \yfs_{(2,0,3,0)} + \yfs_{(1,0,2,2)} + \yfs_{(1,0,3,1)}.\]
\end{ex}

\subsection{A weak insertion algorithm}

We now provide the positive expansion of the dual immaculate slide polynomials into Young quasi-key polynomials.  To do this, we construct a bijection $\phi : \SIF(a)\rightarrow Y(a)$, where $Y(a)$ is the set of pairs $(P,Q)$ such that $P$ is a $\YSF$ and $Q$ is a DIRF (defined below) with row strip shape $\rev(a)$ (also defined below) and the same shape as $P$. This bijection, described below, will preserve weak descent compositions, in the sense that if $\phi(U)=(P,Q)$ then $\wdes_{\SIF}(U) = \wdes_{\YSF}(P)$, which in turn proves that the expansions into fundamental slide polynomials match.  

Let $a$ be a weak composition of $n$. A \emph{recording triple} in $D(a)$ is a collection of three boxes in the positive quadrant with two adjacent in a row, and the third box in the same column as the left box but above it, possibly with rows in between. 

\begin{figure}[ht]
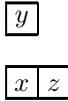

  \begin{displaymath}
    \begin{array}{l}
       \tableau{ y } \\ \\ \tableau{ x  & z }
    \end{array}
  \end{displaymath}
  \caption{Recording triples for weak composition diagrams.}\label{fig:recordingtriples}
\end{figure}

To be a recording triple in $D(a)$, the boxes labeled $x$ and $y$ in Figure~\ref{fig:recordingtriples} must be in $D(a)$, but the box labeled $z$ does not have to be; consider the box labeled $z$ to contain infinity if it is not in $D(a)$.  Given any filling of the boxes of $D(a)$ with $\{1,\ldots , n\}$, a recording triple in $D(a)$ is said to be an \emph{inversion recording triple} if whenever $y>x$ we also have $y>z$. Note that the recording triple rule is different from the triple rule appearing in Definition~\ref{def:YSSF} of Young semi-skyline fillings.

\begin{definition}{\label{def:DIRF}}
Let $b$ be a weak composition of $n$. A filling of $D(b)$ with the entries $\{1,\ldots , n\}$, each used exactly once, is a \emph{dual immaculate recording filling} (or DIRF) if it satisfies the following:
\begin{enumerate}[leftmargin=1.8cm]
\item[(DIRF1)] Entries increase from left to right along rows
\item[(DIRF2)] If $i+1$ is weakly left of $i$, then $i+1$ is in the leftmost column
\item[(DIRF3)] The leftmost column increases from top to bottom
\item[(DIRF4)] All recording triples are inversion recording triples.
\end{enumerate}
If $b$ is in fact a composition, we say instead that such a filling is a DIRT (dual immaculate recording tableau).
\end{definition}

The conditions defining a DIRF are the same as the conditions in~\cite{AHM18} defining a DIRT (on a composition diagram), except that the notion of \emph{row strips} from~\cite{AHM18} is incorporated into (DIRF2) here.  A \emph{row strip} is a maximal sequence of consecutive integers $r_1,r_2, \hdots , r_k$ such that $r_i$ is strictly to the left of $r_{i+1}$ for $1 \le i < k$.  Therefore our definition agrees with the definition of a DIRT from \cite{AHM18} when $b$ is a composition.  The \emph{row strip shape} of a DIRF $Q$ is the weak composition formed as follows. Consider the rows of $Q$ from \emph{top to bottom}. If there is no entry in a given row we record a zero, otherwise we record the length of the row strip beginning in that row. We note that in the case where $b$ is a composition, this also agrees with the definition of row strip shape for DIRTs from \cite{AHM18}, since for a composition diagram there are no empty rows.

\begin{ex}\label{ex:DIRF}
Let $\ell=4$ and $b=(1,0,4,0)$. The filling 
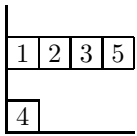
\[\vline\tableau{ \\  1 & 2 & 3 & 5 \\ \\  4  \\\hline }\] 
of $D(b)$ is a DIRF with row strip shape $(0,3,0,2)$.
\end{ex}

Given $U\in \SIF(a)$, recall that its reading word is formed by reading the rows of $U$ from left to right, starting at the top row and proceeding downwards. We define the (weak) insertion of the reading word $u_1u_2 \cdots u_n$ of $U$ using the following procedure, which extends the insertion algorithm of~\cite{AHM18}.  

Let $P_0$ be an empty diagram.  Place the first entry $u_1$ from the reading word of $U$ into the leftmost column of the empty diagram $P_0$ in the row containing $u_1$ in $U$.  This produces the diagram $P_1$ consisting of one box.  Place a ``1" in the corresponding location in $Q_0$ creating $Q_1$.

Assume the first $k$ entries in the reading word of $U$ have been inserted into a diagram called $P_k$ and insert $u_{k+1}$ as follows.  Scan the columns of $P_k$ in reading order as a $\YSF$ (i.e. from top to bottom, right to left), until a cell $(c_i,d_i)$ is reached such that $P_k(c_i-1,d_i) \le u_{k+1} < P_k(c_i,d_i)$.  (If $(c_i,d_i)$ is not in the diagram $P_k$, consider $P_k(c_i,d_i)$ to be equal to infinity.)  Switch the roles of $P_k(c_i,d_i)$ and $u_{k+1}$ (so that $u_{k+1}$ is placed into the cell previously containing $P_k(c_i,d_i)$ and $P_k(c_i,d_i)$ is now used for the scanning).  Continue scanning the reading word of $P_k$ from this point, now seeking an entry larger than $P_k(c_i,d_i)$ immediately to the left of an entry smaller than $P_k(c_i,d_i)$ so that $P_k(c_i,d_i)$ can bump this entry.  Repeat this process until either the entry ``bumped" is not in the leftmost column but is equal to infinity (at which point the procedure terminates) or the leftmost column is reached.  If an entry must be inserted into the leftmost column, place it in the same row in which it appears in $U$.  The resulting diagram is called $P_{k+1}$.  Continue building the recording tableau by placing $k+1$ in the cell where the insertion procedure into $P_k$ terminates, and call the resulting recording diagram $Q_{k+1}$.  (See Example~\ref{ex:2030}.)

\begin{lemma}\label{lem:weakinsert}
Weak insertion with respect to $U\in \SIF(a)$ is well defined, and the resulting tableau $P$ is a $\YSF$ such that $U$ and $P$ have identical leftmost columns. Moreover, the recording tableau $Q$ is a DIRF of the same shape as $P$ and has row strip shape $\rev(a)$.
\end{lemma}
\begin{proof}
Well-definedness of weak insertion follows from well-definedness of the insertion of \cite{AHM18}, except we need to show we can actually place entries in the specified rows in the first column, as opposed to just creating new rows as in the insertion given in \cite[Procedure 3.1]{AHM18}.  Divide the reading word of $U$ into increasing runs from left to right. By definition, the runs correspond to the rows of $U$, and the first entry of each run is smaller than the first entry of the previous run; in particular, it is smaller than every entry that has been inserted so far.  When an increasing sequence is inserted, each entry of this sequence ends up strictly to the right of the preceding entry in the sequence (see~\cite[Lemma 3.5]{AHM18}). Therefore, the insertion procedure will place an entry in the leftmost column if and only if it is the first entry of an increasing run, i.e., an entry from the leftmost column of $U$. Since such an entry is smaller than every entry inserted so far, the corresponding position in the leftmost column is empty at the time of this insertion. Hence we may place the entry in the leftmost box of the row which contains the entry in $U$, and the procedure is well-defined. This also establishes that $U$ and $P$ have identical first column. The fact that $P$ is a $\YSF$ follows from the fact that insertion of the reading word of a standard immaculate tableau yields a standard Young composition tableau (\cite{AHM18}), and the fact that introducing empty rows does not affect any of the increasing or triple conditions.

The nonzero entries of the weak composition that is the row strip shape of $Q$ will be the nonzero entries of $\rev(a)$, since the empty rows of $P$ are precisely the empty rows of $U$, and the row strip shape of $Q$ records these rows in reverse order. It is immediate from the definition that $Q$ has the same shape as $P$. Then the fact that $Q$ is a DIRF and has row strip shape $\rev(a)$ follows from the fact (\cite[Corollary 3.10]{AHM18}) that insertion of standard immaculate tableaux for dual immaculate quasisymmetric functions yields DIRTs of row strip shape $\rev(\alpha)$.
\end{proof}

\begin{ex}{\label{ex:2030}}
Let $a=(2,0,3,0)$. The $\SIF$s for $a$ and their weak descent compositions are in Figure~\ref{fig:2030SIF}. Inserting their reading words gives the $\YSF$s shown in Figure~\ref{fig:YSFfrominsertion}, which we note agree on $\wdes$.
\end{ex}

\begin{figure}[ht]
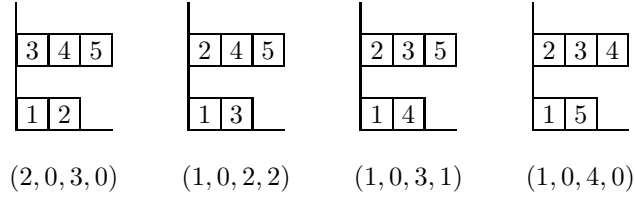

  \begin{center}
    \begin{displaymath}
      \begin{array}{c@{\hskip2\cellsize}c@{\hskip2\cellsize}c@{\hskip2\cellsize}c}
        \vline\tableau{ \\  3 & 4 & 5 \\ \\  1 & 2  \\\hline } &
        \vline\tableau{ \\  2 & 4 & 5 \\ \\  1 & 3  \\\hline } &
        \vline\tableau{ \\  2 & 3 & 5 \\ \\  1 & 4  \\\hline } &
        \vline\tableau{ \\  2 & 3 & 4 \\ \\  1 & 5  \\\hline } \\ \\
        (2,0,3,0) & (1,0,2,2) & (1,0,3,1) & (1,0,4,0)
                     \end{array}
    \end{displaymath}
  \end{center}
  \caption{The 4 $\SIF$s for $a=(2,0,3,0)$ and their weak descent compositions $\wdes_{\SIF}$.}\label{fig:2030SIF}
  \end{figure}

  \begin{figure}[ht]
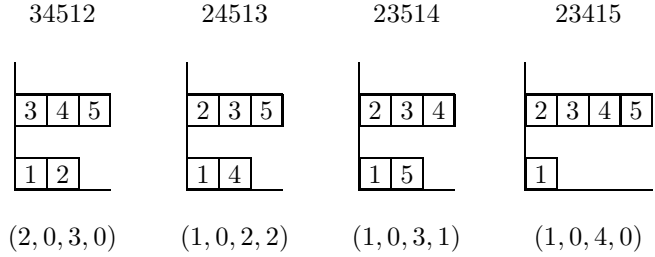

  \begin{center}
    \begin{displaymath}
      \begin{array}{c@{\hskip2\cellsize}c@{\hskip2\cellsize}c@{\hskip2\cellsize}c}
        34512 & 24513 & 23514 & 23415 \\ \\
        \vline\tableau{ \\  3 & 4 & 5 \\ \\  1 & 2  \\\hline } &
        \vline\tableau{ \\  2 & 3 & 5 \\ \\  1 & 4  \\\hline } &
        \vline\tableau{ \\  2 & 3 & 4 \\ \\  1 & 5  \\\hline } &
        \vline\tableau{ \\  2 & 3 & 4 & 5 \\ \\  1   \\\hline } \\ \\
        (2,0,3,0) & (1,0,2,2) & (1,0,3,1) & (1,0,4,0)
                     \end{array}
    \end{displaymath}
  \end{center}
  \caption{The reading words of the $\SIF$s from Figure~\ref{fig:2030SIF} above, the $\YSF$s arising from insertion of those words, and their weak descent compositions $\wdes_{\YSF}$.}\label{fig:YSFfrominsertion}
  \end{figure}
  
    \begin{figure}[ht]
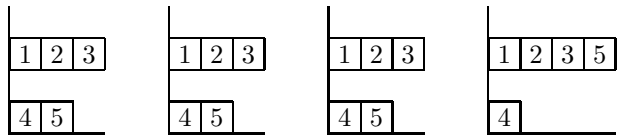

  \begin{center}
    \begin{displaymath}
      \begin{array}{c@{\hskip2\cellsize}c@{\hskip2\cellsize}c@{\hskip2\cellsize}c}
        \vline\tableau{ \\  1 & 2 & 3 \\ \\  4 & 5  \\\hline } &
        \vline\tableau{ \\  1 & 2 & 3 \\ \\  4 & 5  \\\hline } &
        \vline\tableau{ \\  1 & 2 & 3 \\ \\  4 & 5  \\\hline } &
        \vline\tableau{ \\  1 & 2 & 3 & 5 \\ \\  4   \\\hline } \\ \\
                     \end{array}
    \end{displaymath}
  \end{center}
  \caption{The DIRFs from the $\YSF$s in Figure~\ref{fig:YSFfrominsertion} above. All have row strip shape $\rev(2,0,3,0)=(0,3,0,2)$.}\label{fig:2030DIRF}
  \end{figure}

\begin{lemma}\label{lem:wdespreserved}
Let $U\in \SIF(a)$ and let $(P,Q)$ be the $(\YSF, {\rm DIRF})$ pair resulting from insertion of the reading word of $U$. Then $\wdes_{\SIF}(U) = \wdes_{\YSF}(P)$.
\end{lemma}

\begin{proof}
To see that the descent set of the $\SIF$ $U$ is the same as the descent set of the $\YSF$ $P$, first note that every $\YSF$ of shape $a$ collapses (via the removal of all empty rows) to a standard Young composition tableau of shape $\flatten(a)$ (Definition~\ref{def:YCT}). Likewise, every $\SIF$ of shape $a$ collapses to a $\SIT$ of shape $\flatten(a)$ (Definition~\ref{def:SSIT}). In both cases, the positions of empty rows do not affect the descent set, and by \cite{AHM18} the descent set of $U\in \SIT(\alpha)$ is the same as the descent set of the standard Young composition tableau obtained by inserting $U$. So $i$ is a descent in $U$ if and only if $i$ is a descent in $P$. Therefore the run decomposition of $1 \, 2 \ldots  n-1 \, n$ is the same for both $U$ and $P$. 

Now recall (from Definition~\ref{def:wdesdis} of $\wdes_{\SIF}(U)$ and Definition~\ref{def:wdesyqk} of $\wdes_{\YSF}(P)$) that the integers $p_1 < \ldots < p_k$ are obtained recursively by setting $p_j$ equal to the maximum of $p_{j-1}+1$ and the index of the highest row containing an entry from the $j^{th}$ run $r_j$.  We will show by induction on $j$ that $p_j$ is the same in both $U$ and $P$, for all $1\le j \le k$. This will establish that $\wdes_{\SIF}(U) = \wdes_{\YSF}(P)$, including that $\wdes_{\SIF}(U) =\emptyset$ if and only if $\wdes_{\YSF}(P)=\emptyset$. 

First we establish that $p_1$ is the same in both $U$ and $P$. By definition, in both $U$ and $P$, $p_1$ is the index of the row containing the entry $1$. Now, since $1$ is the smallest entry and entries increase along rows of both $U$ and $P$, $1$ is necessarily in the leftmost column of both $U$ and $P$. Since $U$ and $P$ have identical leftmost column by Lemma~\ref{lem:weakinsert}, the index of the row containing $1$ is the same in both $U$ and $P$.

Now suppose that $p_1, \ldots , p_j$ agree on $U$ and $P$ for some $j\ge 1$. We will show that $p_{j+1}$ agrees on $U$ and $P$. 
Suppose $i$ is the last element of the $j^{th}$ run, (in both $U$ and $P$, since $U$ and $P$ have the same run decomposition), so $i+1$ is the first element of the ($j+1$)th run, and suppose $q$ is the last element of the $(j+1)^{th}$ run. Since the leftmost columns of $U$ and $P$ are identical, there are only two cases concerning the location of $i+1$ in $U$ and $P$. Either

\begin{enumerate}
\item $i+1$ is not in the leftmost column of $U$, $P$; or
\item $i+1$ is in the leftmost column of $U$, $P$.
\end{enumerate}

We claim that in case (1), we must have $p_{j+1} = p_j+1$ in both $U$ and $P$. To see this, first note that since $i+1$ is not in the leftmost column of $U$, the row of $U$ containing $i+1$ has a smaller entry to the left of $i+1$.  Therefore, the index of the row containing the entry $i+1$ in $U$ has already been considered in the computation of $p_r$ for some $r<j$, meaning that $p_r$ (and hence $p_j$) for $U$ is at least the index of the row containing the entry $i+1$.  Hence we must have $p_{j+1} = p_j+1$ in $U$.

We now show that $p_{j+1}=p_j+1$ in $P$.  If no entry of the run $i+1, \ldots , q$ of $P$ is in a higher row than the row of $U$ containing $i+1$, we are done.  Otherwise, let $s$ be an entry of the $(j+1)^{th}$ run of $P$ such that all other entries of this run are weakly below $s$.  Since $s$ is in the same run as $i+1$, either $s=i+1$ or $s$ is strictly to the right of the entry $i+1$.  In particular, $s$ may not be in the leftmost column of $P$.  Now, the leftmost (i.e. first column) entry in the row containing $s$ must be smaller than $s$ (by (YSSF1)) and in fact must be less than $i+1$ since it can't be in the run consisting of $i+1, \hdots , q$.  Hence again the row index of the highest row containing an element of the $(j+1)^{th}$ run has been considered already in the computation of some $p_r$ for $r<j$, and thus we must have $p_{j+1} = p_j+1$ in $P$.

Now suppose we are in case (2). Since $i+1$ is in the leftmost column, it is in fact in the same position in both $U$ and $P$, by Lemma~\ref{lem:weakinsert}. In particular, $i+1$ is in the same row in both $U$ and $P$. We know that runs in $U$ proceed weakly downwards, so the highest row containing an entry of the run $i+1, \ldots , q$ in $U$ is the row containing $i+1$. It is therefore enough to show that in $P$, no entry of the run $i+1, \ldots , q$ is in a row strictly higher than the row containing $i+1$, since this will establish that the index of the highest row containing an element of the $(j+1)^{th}$ run $i+1, \ldots , q$ is the same in both $U$ and $P$.  It then follows that since $p_j$ is equal for $U$ and $P$, $p_{j+1}$ is also equal for $U$ and $P$.

To see this, first consider the point in the insertion procedure at which $i+1$ is inserted. Since $i+1$ goes into the leftmost column of $P$, $i+1$ has to be the leftmost entry of some row of $U$. Therefore, since the leftmost columns of $U$ and $P$ agree and since leftmost column entries of $P$ can never be bumped during the insertion procedure, the smallest entry inserted prior to the insertion of $i+1$ is the leftmost entry of the lowest occupied row above the row of $i+1$ in $P$. This entry, say $s$, is strictly greater than $i+1$ since the leftmost column of both $U$ and $P$ increases as entries are read from bottom to top. Now, if we had $s\le q$, then we would have a descent in $P$ between $s-1\ge i+1$ and this first-column entry $s$, contradicting that $i+1, \ldots , q$ is a run in $P$. (Recall the leftmost column entries can't be bumped, so if this leftmost column entry is $s$ at any time it is $s$ forever after that.)  Therefore in $P$, the leftmost entry of every row above the row whose leftmost column entry is $i+1$, must in fact be greater than $q$. Since the leftmost entry of a row can never change during insertion and  entries must increase along rows, those rows can never have any entry from $i+2, \ldots , q$. So indeed in $P$ no entry from the rest of the run $i+2, \ldots , q$ can end up in a row higher that the row containing $i+1$, as required. 
\end{proof}

We are now in a position to prove our positive formula for the Young quasi-key expansion of a dual immaculate slide polynomial. Recall Theorem~\ref{thm:distoyqk} states that for $a$ a weak composition,

\begin{equation}\label{eq:distoyqk}
\dis_a = \sum_b c_{a,b} \yqk_b,
\end{equation}

\noindent\emph{Proof of Theorem~\ref{thm:distoyqk}:} 
Let $Y(a)$ denote the set of pairs $(P,Q)$ such that $P$ is a $\YSF$ and $Q$ is a DIRF of the same shape as $P$ having row strip shape $\rev(a)$. We claim there is a bijection from $\SIF(a)$ to $Y(a)$ satisfying $\wdes_{\SIF}(U) = \wdes_{\YSF}(P)$ whenever $U$ is the $\SIF$ associated to $(P,Q)$ under the bijection. The existence of such a bijection would imply that
\[\sum_{U\in \SIF(a)}\yfs_{\wdes_{\SIF}(U)} = \sum_{(P,Q)\in Y(a)}\yfs_{\wdes_{\YSF}(P)}.\]

The left-hand side of the expression above is equal to the left-hand side of $(\ref{eq:distoyqk})$ by Theorem~\ref{thm:ydistoyslide}, while the right-hand side is equal to the right-hand side of $(\ref{eq:distoyqk})$, by Theorem~\ref{thm:yqktoyslide}. Note that by Lemma~\ref{lem:wdespreserved}, we have $\wdes_{\SIF}(U)=\emptyset$ if and only if $\wdes_{\YSF}(P)=\emptyset$, so there is no complication introduced by the possibility that certain $U$ or $P$ might not yield a fundamental slide polynomial.

To show such a bijection exists, fix $a$ and consider the weak insertion map taking $U\in \SIF(a)$ to $(P,Q)$. By Lemma~\ref{lem:weakinsert} we know that $(P,Q)\in Y(a)$, and by Lemma~\ref{lem:wdespreserved} we know $\wdes_{\SIF}(U) = \wdes_{\YSF}(P)$. To show this is a bijection, note that an inverse can be defined using a variation of the \emph{rapture} procedure of \cite{AHM18} to form a word from $(P,Q)$.  We give an outline of this procedure and then provide details below.

\begin{enumerate}
\item Locate the largest entry in $Q$ and remove it, resulting in $Q'$.
\item Find the entry $r_0$ in the corresponding box of $P$.
\item Reverse the insertion procedure to obtain a new diagram $P'$ and a letter $m$ that was inserted into $P'$ to construct $P$.
\item Repeat using $Q'$ instead of $Q$ and using $P'$ in place of $P$.  Construct a word by appending the letter $m$ obtained at each iteration to the front of the word created by the previous $m$ values.
\end{enumerate}

The rapture process used in Step 3 is identical to the ``rapture" procedure described in~\cite{AHM18} but we describe it here for completeness.  Scan the reading word of $P$ in reverse order, beginning with the entry appearing immediately before $r_0$ in reading order, unless $r_0$ is in the leftmost column.  If $r_0$ is in the leftmost column, start with the last entry in reading order in the second column from the left.  Find the first box containing an entry $r_1$, smaller than $r_0$, lying immediately to the left of an entry greater than or equal to $r_0$.   Place $r_0$ in the cell occupied by $r_1$ and continue scanning the reading word in reverse order and ``bumping" out smaller entries until the last letter (i.e. the first letter in reading order) is scanned.  Set $m$ equal to the last entry bumped.  (See Figure~\ref{fig:rapture} for an example.)

We need to check that the resulting filling $P'$ satisfies the YSSF conditions (1),(2), and (4) from Definition~\ref{def:YSSF}.  First note that (YSSF1) is satisfied by construction.  Since there are no repeated entries in $P$, column entries must be distinct and will remain distinct when bumped and removed, so (YSSF2) is also satisfied.  Finally, we must prove (YSSF4), which states that every Young triple is an inversion triple.  The only way replacing a smaller entry with a larger entry could convert an inversion triple to a non-inversion triple is if in the triple $\tableau{z & x \\ \\ & y}$, the entry $x$ is replaced with an entry $r_i$ greater than $y$ while $y \ge z$.  But in that case, $r_i$ could not have been bumped from a cell in between the cell containing $y$ and the cell containing $x$.  If it were, then $r_i > x$ implies the entry immediately to the left of $r_i$ is less than $z$, and this pattern of the lower row containing a larger entry continues all the way to the leftmost column.  But this contradicts the fact that the leftmost column entries strictly increase from bottom to top.

Since $r_i > y$, we know $y$ did not bump $r_i$.  But then $r_i$ was compared to $y$.  Since $r_i$ did not bump $y$, it must be the case that the entry $e$ immediately to the right of $y$ is less than $r_i$.  We must have $e \ge y$, since the row entries in $P$ increase from left to right, but this implies $e > x$ since $y > x$.  But then $e, x,$ and the entry $f$ immediately to the right of $x$ must form a non-inversion triple in $P$ since $e<f$.  This is because if $e \ge f$ then $r_i > f$ and hence $r_i$ would not bump $x$.  This contradicts condition (YSSF4) for $P$.  Therefore all Young triples remain inversion triples and (YSSF4) is satisfied.

In fact, when $m$ is inserted into $P'$, the resulting filling is $P$.  To see this, note that insertion of an entry $k$, bumped from a given position in reading order, bumps the next entry which is greater than $k$ and left of an entry greater than or equal to $k$.  This is precisely the inverse of the bumping procedure described above. 

We claim the resulting word $m=m_1 m_2 \cdots m_n$ is then the reading word of the unique $U\in \SIF(a)$ (whose leftmost column is identical to the leftmost column of $P$), that is mapped to $(P,Q)$ under insertion.  It is clear from the above paragraph that the word $m$ is in fact the unique word mapping to $(P,Q)$ under insertion.  To see that $m$ is the reading word of an element $U$ in $\SIF(a)$, begin by placing $m_1$ in the row containing the entry $m_1$ in $P$.  (Note that $m_1$ must be equal to the highest entry in the leftmost column of $P$ since the entry $1$ in $Q$ must be in the highest position of the leftmost column of $Q$.) Place the remaining letters of $m$ into this row in order until an entry in the leftmost column of $P$ appears in $m$.  Begin a new row of $U$ by placing this entry into the row corresponding to its position in $P$ and continue.   We must prove that the resulting diagram is in fact a $\SIF$.

The entries in the leftmost column are the same as those in $P$, and therefore (SSIF2) from Definition~\ref{def:SSIF} is satisfied by construction.  Since we don't need condition (SSIF3) for a standard immaculate filling, it is enough to show (SSIF1), that row entries increase from left to right.

Condition (DIRF2) from Definition~\ref{def:DIRF} implies that if $a_1$ and $a_2$ are consecutive elements of $Q$ (i.e.  $a_2=a_1+1$) such that $a_2$ is not in the leftmost column, then $a_1$ is strictly to the left of $a_2$.  In the language of~\cite{AHM18}, this means that $a_1$ and $a_2$ are consecutive elements in the same row strip.  One of the elements of the proof of Procedure 3.20 in~\cite{AHM18} is to show that if $a_1$ and $a_2$ are consecutive elements in the same row strip and the rapture of $a_1$ produces an output of $e_1$ and the rapture of $a_2$ produces the output $e_2$, then $e_1 < e_2$.  Therefore the entries raptured from $P$ - starting with the entry corresponding to the largest entry in $Q$ and moving through the entries of $Q$ in decreasing order - will produce a decreasing sequence of outputs until the leftmost column is reached.  Therefore the entries placed into a row of $U$ will in fact increase from left to right, as desired.
\qed

\begin{figure}[ht]
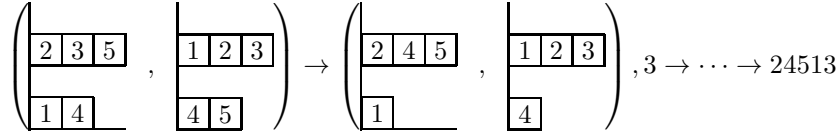

$\left( \vline\tableau{ \\  2 & 3 & 5 \\ \\  1 & 4  \\\hline } \; \; \; {,} \; \; \; \vline\tableau{ \\ 1 & 2 & 3 \\ \\ 4 & 5} \right) \rightarrow \left( \vline\tableau{ \\  2 & 4 & 5 \\ \\  1   \\\hline } \; \; \; {,} \; \; \; \vline\tableau{ \\ 1 & 2 & 3 \\ \\ 4 } \right), 3 \rightarrow \cdots \rightarrow 24513$
\caption{A pair $(P,Q) \in Y(2,0,3,0)$ and the reading word of their associated $\SIF$.}{\label{fig:rapture}}
\end{figure}

Reversing the variables yields a positive formula for the quasi-key expansion of a reverse dual immaculate slide polynomial:

\begin{cor}\label{rdistoqk}
Let $a$ be a weak composition of length $\ell$. Then
\begin{equation}\label{eq:rdistoqk}
\rdis_a = \sum_b c_{\rev(a),\rev(b)} \qkey_b,
\end{equation}
where $c_{a,b}$ is the number of DIRFs of shape $b$ with row strip shape $\rev(a)$.
\end{cor}

\begin{proof}
By Theorem~\ref{thm:distoyqk}, we have
\[\dis_{\rev(a)}(x_\ell, \ldots , x_1) =  \sum_{\rev(b)} c_{\rev(a),\rev(b)} \yqk_{\rev(b)}(x_\ell, \ldots , x_1).\]
By the definitions of the Young versions of the dual immaculate slide polynomials and quasi-key polynomials, the left hand side is equal to $\rdis_a(x_1, \ldots , x_\ell)$.  The right hand side is equal to $\sum_{\rev(b)} c_{\rev(a),\rev(b)} \qkey_b(x_1, \ldots , x_\ell)$, and summing over all $\rev(b)$ is the same as summing over all $b$.
\end{proof}

We may take the stable limit of each side of the above formula to obtain a formula for the quasisymmetric Schur expansion of a reverse dual immaculate function:

\begin{cor}\label{cor:rditoqs}
Let $\alpha$ be a composition. Then 
\[\rdi_{\alpha} = \sum_\beta c_{\rev(\alpha),\rev(\beta)} \qs_{\beta},\] 
where $c_{\alpha,\beta}$ is the number of DIRTs of shape $\beta$ with row strip shape $\rev(\alpha)$.
\end{cor}

\begin{proof}
Taking stable limits of (\ref{eq:rdistoqk}), we obtain 
\[\rdi_{\flatten(a)} = \sum_b c_{\rev(a),\rev(b)} \qs_{\flatten(b)},\] 
where $c_{a,b}$ is the number of DIRFs of shape $b$ with row strip shape $\rev(a)$. 

It remains to observe that for weak compositions $a,b$, we have $c_{a,b} = c_{\flatten(a), \flatten(b)}$. To see this, first note that introducing/removing empty rows into/from a DIRF yields another DIRF: the conditions (DIRF1--4) are unaffected by such introductions/removals. Now suppose $c_{a,b}$ is nonzero. This implies that for each $i$, $a_i$ is nonzero if and only if $b_i$ is nonzero. It is clear that given a DIRF of shape $b$ with row strip shape $\rev(a)$, removing all empty rows yields a DIRT of shape $\flatten(b)$ with row strip shape $\rev(\flatten(a))$. Conversely, given a DIRT of shape $\flatten(b)$ with row strip shape $\rev(\flatten(a))$, one can introduce empty rows in a unique way to obtain a DIRF of shape $b$. Since the diagrams of $a$ and $b$ have the same occupied rows, and row strips always start in the first column, this DIRF will have row strip shape $\rev(a)$. 
\end{proof}

From this, we can recover the formula of \cite{AHM18} given in Theorem~\ref{thm:ditoyqs}.

\begin{cor}\cite[Theorem 1.1]{AHM18}
\[\di_{\alpha} = \sum_\beta c_{\alpha, \beta} \yqs_{\beta},\] 
where $c_{\alpha,\beta}$ is the number of DIRTs of shape $\beta$ with row strip shape $\rev(\alpha)$. 
\end{cor}
\begin{proof}
For any positive integer $\ell$, we have 
\begin{align*}
\di_{\alpha}(x_1, \ldots , x_\ell) & = \rdi_{\rev(\alpha)}(x_\ell , \ldots , x_1) \\
							 & = \sum_\beta c_{\alpha, \beta} \qs_{\rev(\beta)}(x_\ell, \ldots , x_1) \\
							 & = \sum_\beta c_{\alpha, \beta} \yqs_{\beta}(x_1, \ldots , x_\ell)
							 \end{align*}
where the first equality is by Proposition~\ref{prop:ditordi} and the second by Corollary~\ref{cor:rditoqs}. 
The result follows by letting $\ell \rightarrow \infty$.
\end{proof}

As a final corollary, we establish a connection between the dual immaculate basis and the \emph{Demazure atom} basis introduced by Lascoux-Sch\"utzenberger~\cite{LasSch90} and further developed in~\cite{Mas09}.  These polynomials are obtained as specializations of nonsymmetric Macdonald polynomials, and can be interpreted representation-theoretically as characters of quotients of Demazure modules.   To see the connection to the dual immaculate basis, note that the reverse dual immaculate slide polynomials and the reverse dual immaculate quasisymmetric functions expand positively in the basis of Demazure atoms. 

\begin{cor}{\label{cor:rdistodem}}
The reverse dual immaculate slide polynomials and the reverse dual immaculate quasisymmetric functions expand positively in Demazure atoms.
\end{cor}
\begin{proof}
For the reverse dual immaculate slide polynomials, this follows from Corollary~\ref{rdistoqk} and the fact that quasi-key polynomials expand positively in Demazure atoms \cite{Sea}. For the reverse dual immaculate quasisymmetric functions, this follows from Corollary~\ref{cor:rditoqs} and Theorem~\ref{thm:ditoyqs}, since quasisymmetric Schur functions expand positively in Demazure atoms \cite{HLMvW11a}.
\end{proof}

Recall that the dual immaculate quasisymmetric functions can can be obtained from the reverse dual immaculate quasisymmetric functions via a reversal of both the variable set and the composition parts.  Therefore Corollary~\ref{cor:rdistodem} implies that the dual immaculate quasisymmetric functions expand positively in Demazure atoms up to a similar reversal procedure.

\appendix
\section{Families  of functions and polynomials and their corresponding tableaux appearing in this paper}{\label{appendix}}

\begin{figure}[h]
\begin{center}
\begin{tabular}{|p{1.6in}|p{2in}|p{2in}|}
\hline
{\bf Object Name and \newline Function Generated} & {\bf Basic Properties} & {\bf  Standard version, reading order, and descent set} \\\hline\hline
Semistandard reverse immaculate tableau ($\SSRIT$, Def.~\ref{def:SSRIT}) \newline \newline \framebox{$\rdi_{\alpha}$} \newline
Reverse dual immaculate quasisymmetric function & 
\begin{enumerate}[align=left,style=nextline,leftmargin=*]
\item Row entries weakly decrease left to right.
\item Leftmost column entries strictly increase bottom to top.
\end{enumerate} & 
\begin{itemize}[align=left,style=nextline,leftmargin=*]
\item Standard ($\SRIT$):  each of $\{1,2, \hdots , n\}$ appears once
\item Reading order: rows from right to left, top to bottom 
\item $\des_{\SRIT} = \{ i$ such that $i+1$ is in a strictly higher row $\}$
\end{itemize}  \\
\hline
Semistandard immaculate tableau ($\SSIT$, Def.~\ref{def:SSIT}) \newline \newline \framebox{$\di_{\alpha}$} \newline
Dual immaculate quasisymmetric function & 
\begin{enumerate}[align=left,style=nextline,leftmargin=*]
\item Row entries weakly increase left to right.
\item Leftmost column entries strictly increase bottom to top.
\end{enumerate} & 
\begin{itemize}[align=left,style=nextline,leftmargin=*]
\item Standard ($\SIT$):  each of $\{1,2, \hdots , n\}$ appears once
\item Reading order: rows from left to right, top to bottom
\item $\des_{\SIT} = \{ i$ such that $i+1$ is in a strictly higher row $\}$
\end{itemize}
\\
\hline
 \framebox{$\qs_{\alpha}$} \newline
Quasisymmetric Schur function & 
Not used in this paper
& 
Not used in this paper
\\
\hline
Young composition tableau ($\YCT$, Def.~\ref{def:YCT}) \newline \newline \framebox{$\yqs_{\alpha}$} \newline
Young quasisymmetric Schur function & 
\begin{enumerate}[align=left,style=nextline,leftmargin=*]
\item Row entries weakly increase left to right.
\item Leftmost column entries strictly increase bottom to top.
\item All Young triples are inversion triples (Figure~\ref{fig:triples}).
\end{enumerate} & 
\begin{itemize}[align=left,style=nextline,leftmargin=*]
\item Standard ($\SYCT$):  each of $\{1,2, \hdots , n\}$ appears once
\item Reading order: columns from top to bottom, right to left
\item $\des_{\yqs} = \{ i$ such that $i+1$ is weakly left of $i$ $\}$
\end{itemize}

 \\\hline

\end{tabular}
\end{center}
\caption{Families of quasisymmetric functions}{\label{table:qsym}}
\end{figure}

\begin{figure}
\begin{center}
\begin{tabular}{|p{1.6in}|p{2in}|p{2in}|}
\hline
{\bf Object Name and \newline Polynomial Generated} & {\bf Basic Properties} & {\bf  Standard version, reading order, and descent set}  \\\hline\hline
Semistandard reverse immaculate filling ($\SSRIF$, Def.~\ref{def:SSRIF}) \newline \newline \framebox{$\rdis_{a}$} \newline
Reverse dual immaculate slide polynomial & 
\begin{enumerate}[align=left,style=nextline,leftmargin=*]
\item Row entries weakly decrease left to right.
\item Leftmost column entries strictly increase bottom to top.
\item Entries in the $i^{th}$ row from the bottom do not exceed $i$.
\end{enumerate} & 
\begin{itemize}[align=left,style=nextline,leftmargin=*]
\item Standard ($\SRIF$):  each of $\{1,2, \hdots , n\}$ appears once, only need conditions (1) and (2)
\item Reading order: rows from right to left, top to bottom
\item $\des_{\SRIF} = \{  i \textrm{ such that } i+1 \textrm{ is in a strictly higher row} \}$
\end{itemize} \\
\hline
Semistandard immaculate filling ($\SSIF$, Def.~\ref{def:SSIF}) \newline \newline \framebox{$\dis_{a}$} \newline
Dual immaculate slide polynomial & 
\begin{enumerate}[align=left,style=nextline,leftmargin=*]
\item Row entries weakly increase left to right.
\item Leftmost column entries strictly increase bottom to top.
\item Entries in the $i^{th}$ row from the bottom do not exceed $i$.
\end{enumerate} & 
\begin{itemize}[align=left,style=nextline,leftmargin=*]
\item Standard ($\SIF$):  each of $\{1,2, \hdots , n\}$ appears once,~only need conditions (1) and (2)
\item Reading order: rows from left to right, top to bottom
\item $\des_{\SIF} = \{ i \textrm{ such that } i+1 \textrm{ is in a strictly higher row} \}$
\end{itemize} \\
\hline
Reverse semistandard skyline filling ($\RSSF$, Def~\ref{def:RSSF}) \newline \newline \framebox{$\qkey_{a}$} \newline
Quasi-key polynomial & 
\begin{enumerate}[align=left,style=nextline,leftmargin=*]
\item Row entries weakly decrease left to right.
\item Entries strictly increase up the first column, and entries in any column are distinct.
\item No entry in row $i$ is greater than $i$.
\item All triples are inversion triples (Figure~\ref{fig:reversetriples}).
\end{enumerate} & 
\begin{itemize}[align=left,style=nextline,leftmargin=*]
\item Standard ($\RSF$):  each of $\{1,2, \hdots , n\}$ appears once, only need conditions (1), (2), (4)
\item Reading order: columns from bottom to top, right to left
\item $\des_{\RSF} = \{ i \textrm{ such that } i+1 \textrm{ is weakly right of } i \}$
\end{itemize}
\\
\hline
Young semi-skyline filling ($\YSSF$, Def.~\ref{def:YSSF}) \newline \newline \framebox{$\yqk_{a}$} \newline
Young quasi-key polynomial & 
\begin{enumerate}[align=left,style=nextline,leftmargin=*]
\item Row entries weakly increase left to right.
\item Entries strictly increase up the first column, and entries in any column are distinct.
\item No entry in row $i$ is smaller than $i$.
\item All Young triples are inversion triples (Figure~\ref{fig:triples}).
\end{enumerate} & 
\begin{itemize}[align=left,style=nextline,leftmargin=*]
\item Standard ($\YSF$):  each of $\{1,2, \hdots , n\}$ appears once, only need conditions (1), (2), (4)
\item Reading order: columns from top to bottom, right to left
\item $\des_{\YSF} = \{ i \textrm{ such that } i+1 \textrm{ is weakly left of } i  \}$
\end{itemize}
\\
\hline
Fundamental semistandard skyline filling ($\FSSF$, Def.~\ref{def:FSSF}) \newline \newline \framebox{$\fs_{a}$} \newline
Fundamental slide polynomial & 
\begin{enumerate}[align=left,style=nextline,leftmargin=*]
\item Row entries weakly decrease left to right.
\item Each entry is greater than all entries in lower rows.
\item Entries in the $i^{th}$ row from the bottom do not exceed $i$.
\end{enumerate} & 
Not used in this paper
\\
\hline

\framebox{$\yfs_{a}$} \newline
Young fundamental slide polynomial &
Not used in this paper &
Not used in this paper

\\  \hline

\end{tabular}
\end{center}
\caption{Families of polynomials}{\label{fig:polytable}}
\end{figure}

\bibliographystyle{alpha}
\bibliography{DIS2019}
\label{sec:biblio}

% ------------------------------------------------------------------------
\end{document}